\documentclass[11pt]{amsart}

\usepackage[utf8]{inputenc}
\usepackage{fontenc}
\usepackage{amsfonts}
\usepackage{amssymb}
\usepackage{amsmath}
\usepackage{amsthm}\usepackage{mathtools}
\usepackage{enumerate}
\usepackage{hyperref}\usepackage{enumitem}
\usepackage{mathrsfs}
\usepackage{tikz}
\usepackage{marginnote}

\newcommand{\forces}{\Vdash}

\theoremstyle{Case1}

\theoremstyle{Case2}

\newcommand{\zet}{\mathbb{Z}}

\newtheorem*{rigprob*}{Rigidity Problem for Uniform Roe Algebras}
\newtheorem*{rigprobcorona*}{Rigidity Problem for Uniform Roe Coronas}

\newcommand{\ZFC}{\mathrm{ZFC}}

\newcommand{\CH}{\mathrm{CH}}
\newcommand{\MA}{\mathrm{MA}}

\newcommand{\PFA}{\mathrm{PFA}}

\newtheorem{theorem}{Theorem}[section]
\newtheorem*{theorem*}{Theorem}
\newtheorem{proposition}[theorem]{Proposition}

\newtheorem*{proposition*}{Proposition}
\newtheorem{lemma}[theorem]{Lemma}
\newtheorem*{lemma*}{Lemma}
\newtheorem{corollary}[theorem]{Corollary}
\newtheorem*{corollary*}{Corollar}

\newtheorem*{fact*}{Fact}
\theoremstyle{definition}
\newtheorem{definition}[theorem]{Definition}
\newtheorem*{definition*}{Definition}
\newtheorem{claim}[theorem]{Claim}
\newtheorem*{claim*}{Claim}

\newtheorem*{conjecture*}{Conjecture}

\newtheorem{theoremi}{Theorem}

\theoremstyle{remark}

\newtheorem*{notation*}{Notation}

\newtheorem*{example*}{Example}

\newtheorem*{remark*}{Remark}

\newtheorem*{note*}{Note}
\newtheorem*{question*}{Question}

\DeclareMathOperator{\cof}{cof}
\DeclareMathOperator{\dom}{dom}
\DeclareMathOperator{\supp}{supp}

\DeclareMathOperator{\im}{im}

\newcounter{my_enumerate_counter}
\newcommand{\pushcounter}{\setcounter{my_enumerate_counter}{\value{enumi}}}
\newcommand{\popcounter}{\setcounter{enumi}{\value{my_enumerate_counter}}}

\usepackage{enumitem}

\begin{document}

\title{Non-vanishing higher derived limits}%
\author[B. Veli\v{c}kovi\'c]{Boban Veli\v{c}kovi\'c}
\address[B. Veli\v{c}kovi\'c]{
Institut de Math\'ematiques de Jussieu - Paris Rive Gauche (IMJ-PRG)\\
Universit\'e de Paris\\
B\^atiment Sophie Germain\\
8 Place Aur\'elie Nemours \\ 75013 Paris, France}
\email{boban@math.univ-paris-diderot.fr}
\urladdr{http://www.logique.jussieu.fr/~boban/}
\author[A. Vignati]{Alessandro Vignati}
\address[A. Vignati]{
Institut de Math\'ematiques de Jussieu - Paris Rive Gauche (IMJ-PRG)\\
Universit\'e de Paris \\
B\^atiment Sophie Germain\\
8 Place Aur\'elie Nemours \\ 75013 Paris, France}
\email{alessandro.vignati@imj-prg.fr}
\urladdr{http://www.automorph.net/avignati}

\begin{abstract}
In the study of strong homology Marde\v{s}i\'c and Prasolov isolated a certain inverse system of abelian groups $\mathbf A$ indexed by elements of $\omega^\omega$. They showed that if strong homology is additive on a class of spaces containing closed subsets of Euclidean spaces then the higher derived limits $\lim^n \mathbf A$ must vanish, for $n>0$. They also proved that under the Continuum Hypothesis $\lim^1 \mathbf A \neq 0$. The question whether $\lim^n \mathbf A$ vanishes, for $n>0$, has attracted considerable interest from set theorists. Dow, Simon and Vaughan showed that under PFA $\lim^1 \mathbf A =0$. Bergfalk show that it is consistent that $\lim^2\mathbf A$ does not vanish. Later Bergfalk and Lambie-Hanson showed that, modulo a weakly compact cardinal, it is relatively consistent with $\ZFC$ that $\lim^n \mathbf A =0$, for all $n$. The large cardinal assumption was recently removed by Bergfalk, Hru\v{s}ak and Lambie-Henson. We complete the picture by showing that, for any $n>0$, it is relatively consistent with $\ZFC$ that $\lim^n \mathbf A \neq 0$. 
\end{abstract}

\keywords{derived limits, additivity, strong homology, weak diamond}
\subjclass[2010]{Primary: 03E35, 03E75, 18E25, 55Nxx}
\maketitle

\maketitle

\section{Introduction}
We consider certain directed systems of abelian groups, their derived inverse limits, and how these limits depend on the additional hypotheses of set theory. Derived inverse limits associated to directed systems are often studied from a category theory point of view (see e.g., Marde\v{s}i\'c's book \cite{Mardesic.Book}) for their connections with strong homology, and consequently with algebraic topology. A recent work of Bergfalk (\cite{Bergfalk.Alephs}) further developed connections with phenomena in infinite combinatorics such as walks on ordinals (after Todor\v{c}evi\'c' \cite{Todorcevic2007}), and consequently $C$-sequences and traces. This work allowed him to glimpse at multidimensional versions of Todor\v{c}evi\'c's minimal walks.

We are primarily interested in the inverse system $\mathbf A$ and its derived inverse limits ${\lim}^n\mathbf A$, for $n> 0$. (For all technical definitions, see \S\ref{sec:prel} and \S\ref{S.A}). 
%We want to understand whether ${\lim}^n\mathbf A$ vanishes or not; this was linked with the existence of certain coherent nontrivial families of functions in $\NN^\NN$, where coherence and triviality are intended `modulo finite'. 
The original motivation for studying ${\lim}^n\mathbf A$ is in that additivity of strong homology on certain topological spaces implies that ${\lim}^n\mathbf A$ must vanish for each $n>0$, see \cite[Theorem 8]{MardesicPrasolov}. It was also shown in \cite{MardesicPrasolov} that the Continuum Hypothesis $\CH$ implies that ${\lim}^1\mathbf A\neq 0$. The hypothesis of $\CH$ was weakened to that of $\mathfrak d$, the dominating number, being $\aleph_1$, see \cite[Theorem 2.4]{Dow-Simon-Vaughan}. In the same paper, it was shown that Shelah's Proper Forcing Axiom ($\PFA$) implies that ${\lim}^1\mathbf A=0$. Hence the vanishing of ${\lim}^1\mathbf A$ is independent of the axioms of $\ZFC$. In general, Goblot's Theorem (see Theorem~\ref{thm:goblot}) implies that if $\mathfrak d=\aleph_n$, then ${\lim}^{i}\mathbf A=0$ for all $i\geq n+1$. In 2017 Bergfalk (\cite{Bergfalk2017}) showed that it is consistent that ${\lim}^2\mathbf A$ does not vanish. Later Bergfalk and Lambie-Henson \cite{BergfalkLH} showed that it is relatively consistent, modulo a weakly compact cardinal, that all the ${\lim}^n\mathbf A$, for $n\geq 1$, simultaneously vanish. The use of a weakly compact cardinal was later eliminated by Bergfalk, Hru\v{s}ak and Lambie-Henson in \cite{Bergfalk2021simutaneously}. Extending \cite{BergfalkLH}, Bannister, Bergfalk and Moore (\cite{BanBergMoore}) showed that, modulo a weakly compact cardinal, it is consistent that strong homology for locally compact separable metric spaces is additive. It is not clear if any large cardinal assumptions are needed for this result. Whether higher derived inverse limits of $\mathbf A$ could be non-vanishing remained an open question, which we answer in our main result.

\begin{theoremi}
For every $n\geq 1$, it is relatively consistent with $\ZFC$ that ${\lim}^n\mathbf A\neq 0$. 
\end{theoremi}

In order to prove this result we isolate a combination of combinatorial principles that together imply that $\lim^n \mathbf A \neq 0$. One of these principles is a version of weak diamond ${\rm w}\diamondsuit$ first studied by Devlin and Shelah \cite{Devlin-Shelah}. This principle has consequences similar to those of the usual diamond principle, but is compatible with a wide range cardinal arithmetic behaviour. It was shown in \cite{Devlin-Shelah} that ${\rm w}\diamondsuit(\omega_1)$ is equivalent to $2^{\aleph_0} < 2^{\aleph_1}$. A similar proof show that if $\kappa$ is regular then ${\rm w}\diamondsuit(\kappa^+)$ is equivalent to $2^\kappa < 2^{\kappa^+}$. For our result we require that weak diamond holds on a specific stationary subset of $\omega_{k+1}$, for all $k < n$. Namely, let $S^k_{k+1}= \{ \alpha < \omega_{k+1}\colon {\rm cof}(\alpha) = \omega_k \}$. We show that if $\mathfrak b$ and $\mathfrak d$ are both equal to $\aleph_n$, and ${\rm w}\diamondsuit(S^k_{k+1})$ holds, for every $k < n$, then $\lim^n \mathbf A \neq 0$. We also show that the combination of these combinatorial principles is relatively consistent with $\ZFC$, for any given $n>0$.

The paper is organized as follows. First, in \S\ref{sec:prel} we introduce the main objects of interest, directed systems of abelian groups and their derived inverse limits. As these objects were often treated in a category theory setting, we spend some time to prove some of the classical results about derived inverse limits from a set theoretic point of view. In particular we prove the flasque version of Goblot's theorem Theorem~\ref{thm:goblot}), and a result of B. Mitchell, Theorem~\ref{thm:mitchell}, saying that the cohomological dimension of $\aleph_n$ is $n+1$, or, in our terms, that there is a directed system $\mathbf G_n$ of cofinality $\aleph_n$ and such that ${\lim}^{n+1}\mathbf G_n\neq 0$. In \S\ref{S.A} we focus on $\mathbf A$ and related inverse systems, and we prove that under the above combinatorial principles $\lim^n \mathbf A$, for $n\geq 1$, does not vanish. Finally, in \S\ref{S.Cons} we show that, for every $n\geq 1$, the combination of the principles required for our main theorem is relatively consistent with $\ZFC$. 

%We conclude the paper with a discussion on what is left to do.

\section{Preliminaries: directed systems and their derived inverse limits}\label{sec:prel}

Let $(\Lambda,\leq)$ be an ordered set, and let $n\geq 0$. We write $\Lambda^{(n)}$ for the set of $\leq$-ordered $n+1$-uples in $\Lambda$. We often confuse $\Lambda^{(0)}$ with $\Lambda$. If $\bar\lambda=(\lambda_0,\ldots, \lambda_{n})\in \Lambda^{(n)}$, $\mu \in\Lambda$, and $i<n$, we write:

\begin{itemize}
\item $\lambda_i$ for the $i$-th element of $\bar \lambda$,
\item $\mu \leq\bar\lambda$ if $\mu \leq\lambda_0$,
\item $\bar\lambda\leq\mu$ if $\lambda_{n}\leq\mu$, and
\item $\bar\lambda^i$ for the sequence $(\lambda_0,\ldots,\lambda_{i-1},\lambda_{i+1},\ldots,\lambda_{n})$ in $\Lambda^{(n-1)}$.
\end{itemize}

\begin{definition}
Let $\Lambda$ be a directed set. We say that 
\[
\mathbf G=\{G_\lambda; p^\mu_\lambda \colon \lambda\leq\mu\in\Lambda\}
\]
is a $\Lambda$-{\em inverse system of abelian groups} if
\begin{itemize}
\item $G_\lambda$ is an abelian group, for all $\lambda\in\Lambda$,
\item $p^\mu_\lambda\colon G_\mu\to G_\lambda$ is a group homomorphism, for all $\lambda \leq \mu\in\Lambda$,
\item if $\lambda\leq\mu\leq\nu$ then $p^\nu_\lambda=p_\lambda^\mu\circ p^\nu_\mu$.
\end{itemize}
If each $p^\mu_\lambda$ is surjective we say that $\mathbf G$ is {\em flasque}. If $M$ is a directed subset of $\Lambda$ we will write $\mathbf G \restriction M$ for the system $\{ G_\lambda; p^\mu_\lambda \colon \lambda\leq\mu\in M \}$.
\end{definition}

Let $\mathbf G$ be a $\Lambda$-inverse system of abelian groups. We will always assume that the individual groups $G_\lambda$ are pairwise disjoint. Let $G = \bigcup_\lambda G_\lambda$. We let ${\rm Fn}(\Lambda^{(n)}, G)$ denote the set of all function from $\Lambda^{(n)}$ to $G$. If $\bar x \in {\rm Fn}(\Lambda^{(n)}, G)$ we write $x_{\bar \lambda}$ for $\bar x (\bar \lambda)$, where $\bar \lambda\in \Lambda^{(n)}$. For $n\geq 0$, define 
\[
\mathbf G^{(n)}=\{\bar x \in {\rm Fn}(\Lambda^{(n)}, G)\colon x _{\bar \lambda} \in G_{\lambda_0}, \mbox{ for all } \bar \lambda \in \Lambda^{(n)} \}.
%=(x_{\bar\lambda})\mid \bar\lambda=(\lambda_0,\ldots,\lambda_n)\in \Lambda^{(n)}, \, x_{\bar\lambda}\in G_{\lambda_0}\}.
\]
Then $\mathbf G^{(n)}$ is an abelian group with operations defined pointwise. Define
\[
\delta^n\colon \mathbf G^{(n-1)}\to \mathbf G^{(n)}
\]
 by letting
\[
\delta^n(\bar x)_{\bar\lambda}=\sum_{1\leq i\leq n}(-1)^ix_{\bar\lambda^i}+p^{\lambda_1}_{\lambda_0}(x_{\bar\lambda^0}),
\]
%\medskip
\noindent for every $\bar\lambda\in \Lambda^{(n)}$. By convention, we set $\mathbf G^{(-1)}=0$ and so $\delta^0=0$. Note that $\delta^{n+1}\circ\delta^n=0$, for every $n\geq 0$. These maps are known as {\em coboundary maps}.
%Most of this paper revolves on studying the following definition.

\begin{definition}
Let $n\geq 1$. An element $\bar x\in \mathbf G^{(n)}$ is {\em coherent} if $\delta^{n+1}(\bar x)=0$. We say that $\bar x\in\mathbf G^{(n)}$ is {\em trivial} if there is $\bar y\in \mathbf G^{(n-1)}$ such that $\delta^{n}(\bar y)=\bar x$. In this case we say that $\bar y$ {\em trivializes} $\bar x$.
\end{definition}

%The following Lemma will be needed
%\begin{lemma}
%Suppose that $\bar x\in \mathbf G^{(n)}$ is coherent, and that $\bar y\in \mathbf G^{(n-1)}$ trivializes $\bar x$. 
%Suppose that $\bar z\in\mathbf G^{(n-1)}$ is coherent. Then $\bar y+\bar z$ trivializes $x$.
%\end{lemma}
For $n\geq 0$, let
\[
{\lim}^{n}\mathbf G=\ker (\delta^{n+1})/\im (\delta^n).
\]
%\medskip
\noindent This is the $n$-th derived inverse limit of $\mathbf G$. Since $\delta^0=0$, ${\lim}^0 \mathbf G$ coincides with $\ker(\delta^1)$, which in turns corresponds to the set of coherent sequences in $\mathbf G^{(0)}$, which is just the inverse limit of $\mathbf G$. In other words
\[
{\lim}^0 \mathbf G=\{ \bar x \in\prod_{\lambda \in \Lambda} G_\lambda\colon \forall \, \lambda\leq\mu \, (x_\lambda=p^\mu_\lambda(x_\mu))\}.
\]
For $n\geq 1$, ${\lim}^n \mathbf G\neq 0$ if and only if there is a coherent nontrivial element of $\mathbf G^{(n)}$.

If $\Lambda$ is an ordered set, a subset $C\subseteq\Lambda$ is {\em cofinal} if for all $\lambda\in\Lambda$ there is $\mu\in C$ with $\lambda\leq \mu$. We denote by $\cof(\Lambda)$ the smallest size of a cofinal subset of $\Lambda$. If $\Lambda$ has a maximum we write $\cof(\Lambda)=\aleph_{-1}$. If $\mathbf G$ is a $\Lambda$-inverse system of abelian groups, we write $\cof(\mathbf G)$ for $\cof(\Lambda)$. Goblot's Theorem gives a constraint on ${\lim}^n\mathbf G$ in case $\cof(\mathbf G)$ is small. (A proof for set theorists can be found in \cite[\S13]{Mardesic.Book}.)

\begin{theorem}[Goblot's Theorem, \cite{Goblot}]\label{thm:goblot}
Let $n\geq -1$. If $\mathbf G$ is an inverse system of abelian groups with $\cof(\mathbf G)=\aleph_n$, then ${\lim}^{i}\mathbf G=0$, for all $i\geq n+2$.\qed
\end{theorem}

Goblot's theorem can be strengthened in case of flasque systems. We present a self-contained proof for completeness. 

\begin{theorem}[Flasque Goblot's Theorem]\label{thm:flasquegoblot}
Let $n\geq 0$. If $\mathbf G$ is a flasque inverse system of abelian groups with $\cof(\mathbf G)=\aleph_n$. Then ${\lim}^k \mathbf G=0$, for all $k\geq n+1$.
\end{theorem}
\begin{proof}
We prove the theorem by induction. 

First, the base case. Let $\mathbf G$ be a flasque $\Lambda$-inverse system for a directed set of countable cofinality $\Lambda$. We need to prove that ${\lim}^1\mathbf G=0$. By directedness of $\Lambda$, we can find an increasing sequence $(\mu_i)_{i\in\omega}$ which is cofinal in $\Lambda$. Let $\bar z\in\mathbf G^{(1)}$ be coherent, i.e., $\delta^2(\bar z)=0$. Define $x_{\mu_0}$ to be any element of $G_{\mu_0}$. If $x_{\mu_n}$ has been defined, by surjectivity of $p_{\mu_n}^{\mu_{n+1}}$, we can find $x_{\mu_{n+1}}\in G_{\mu_{n+1}}$ such that 
\[
p_{\mu_n}^{\mu_{n+1}}(x_{\mu_{n+1}})= x_{\mu_n} - z_{\mu_n\mu_{n+1}}.
\]
An easy calculation by induction on $j-i$ gives that if $i\leq j$ then
\[
z_{\mu_i\mu_j}=x_{\mu_i}-p_{\mu_i}^{\mu_{j}}(x_{\mu_{j}}).
\]
For $\lambda\in\Lambda$, let $i$ be the least such that $\mu_i \geq \lambda$, and set 
\[
x_\lambda=p_{\lambda}^{\mu_i}(x_{\mu_i})+z_{\lambda\mu_i}.
\]
We want to check that $\delta^1(\bar x)=\bar z$. Let $(\lambda_0,\lambda_1)\in \Lambda^{(1)}$. Let $i$ be the least such that $\mu_i \geq \lambda_0$, and $j$ is the least such that $\mu_j \geq \lambda_1$. Then 
\begin{eqnarray*}
z_{\lambda_0\lambda_1}&=&z_{\lambda_0\mu_j}-p_{\lambda_0}^{\lambda_1}(z_{\lambda_1\mu_j})=z_{\lambda_0\mu_i}+p_{\lambda_0}^{\mu_i}(z_{\mu_i\mu_j})-p_{\lambda_0}^{\lambda_1}(x_{\lambda_1}-p_{\lambda_1}^{\mu_j}(x_{\mu_j}))\\
&=&x_{\lambda_0}-p_{\lambda_0}^{\mu_i}(x_{\mu_i})+p_{\lambda_0}^{\mu_i}(x_{\mu_i}-p_{\mu_i}^{\mu_{j}}(x_{\mu_{j}}))-p_{\lambda_0}^{\lambda_1}(x_{\lambda_1})+p_{\lambda_0}^{\mu_j}(x_{\mu_j})\\&=&x_{\lambda_0}-p_{\lambda_0}^{\lambda_1}(x_{\lambda_1}).
\end{eqnarray*}

We are left with the inductive step. Fix $\Lambda$ be a directed set of cofinality $\aleph_n$, and $\mathbf G$ be a flasque $\Lambda$-inverse system of abelian groups. Assume we have proved the theorem for all flasque systems of cofinality $<\aleph_n$. Write $\Lambda=\bigcup_{\alpha<\omega_n}\Lambda_\alpha$, where each $\Lambda_\alpha$ is a directed set with $\cof(\Lambda_\alpha)<\aleph_n$ and, if $\beta$ is a limit, then $\Lambda_\beta=\bigcup_{\alpha<\beta}\Lambda_\alpha$. Fix $\bar z\in \mathbf G^{(n+1)}$ such that $\delta^{n+2}(\bar z)=0$. We want to find $\bar x\in \mathbf G^{(n)}$ such that $\delta^{n+1}(\bar x)=\bar z$. We define inductively $\bar x^\alpha\in \mathbf G \restriction \Lambda_\alpha^{(n)}$ such that 
\[
\text { if }\alpha<\beta\text{ then } \bar x^\beta\restriction \Lambda_\alpha^{(n)}=\bar x^\alpha,
\]
and
\[
\delta^{n+1}(\bar x^\alpha)=\bar z\restriction \Lambda_\alpha^{(n+1)}.
\]
If this can be done, setting $\bar x=\bigcup_\alpha\bar x^\alpha$ would complete the proof of the theorem.

For every $\alpha<\omega_n$, we have $\cof(\Lambda_\alpha)<\aleph_n$, hence by induction there is $\bar y^\alpha\in \mathbf G \restriction \Lambda_\alpha ^{(n)}$ such that $\delta^{n+1}(\bar y^\alpha)=\bar z\restriction \Lambda_\alpha^{(n+1)}$. Set $\bar x^0=\bar y^0$. Suppose $\beta$ is a limit ordinal and $\bar x^\alpha$ has been defined for all $\alpha < \beta$. Since $\Lambda_\beta=\bigcup_{\alpha<\beta}\Lambda_\alpha$, we have that $\Lambda_\beta^{(n)}=\bigcup_{\alpha<\beta} \Lambda_\alpha^{(n)}$, hence we can define $\bar x^\beta=\bigcup_{\alpha<\beta}\bar x^\alpha$. 

We are left with the successor step. Notice that both $\bar x^\alpha$ and $\bar y^{\alpha+1}\restriction \Lambda_\alpha^{(n)}$ trivialize $\bar z\restriction \Lambda_\alpha^{(n+1)}$, hence $\bar y^{\alpha+1}\restriction \Lambda_\alpha^{(n)}-\bar x^\alpha$ is a coherent element of $\mathbf G \restriction \Lambda_\alpha^{(n)}$. By the inductive hypothesis and the fact that $\cof(\Lambda_\alpha)<\aleph_n$, there is $\bar w^\alpha\in \mathbf G\restriction \Lambda_\alpha ^{(n-1)}$ trivialising $\bar y^{\alpha+1}\restriction \Lambda_\alpha^{(n)}-\bar x^\alpha$. We extend $\bar w^\alpha$ to $\bar w^{\alpha+1}\in \mathbf G\restriction \Lambda_{\alpha+1}^{(n-1)}$ by setting:
\[
\bar w^{\alpha+1}(\lambda_0, \ldots, \lambda_{n-1})=
\begin{cases}
\bar w^{\alpha}(\lambda_0,\ldots \lambda_{n-1}) &\text{ if } (\lambda_0, \ldots, \lambda_{n-1}) \in \Lambda_\alpha^{(n-1)}\\
0&\text{ otherwise.}
\end{cases}
\]
Define 
\[
\bar x^{\alpha+1}=\bar y^{\alpha+1}-\delta^n(\bar w^{\alpha+1}).
\]
Then $\bar x^{\alpha+1}\restriction \Lambda_\alpha^{(n)}=\bar x^\alpha$, and 
\[
\delta^{n+1}(\bar x^{\alpha+1})=\delta^{n+1}(\bar y^{\alpha+1}) - (\delta^{n+1}\circ\delta^n )(\bar w^{\alpha+1})=\bar z\restriction \Lambda_{\alpha+1}^{(n+1)}.\qedhere
\]
\end{proof}
The proof of the original Goblot's theorem is similar. One first shows that $\lim^2$ vanishes for systems of countable cofinality and applies the induction as above. 

The following is known as Mitchell's Theorem, and shows that Goblot's bound is optimal. It was proved first in \cite{Mitchell72}, (see also \cite[Theorem 13.11]{Mardesic.Book}), but in all proofs we could find there was not an explicit example of a system as below.

\begin{theorem}[Mitchell, \cite{Mitchell72}]\label{thm:mitchell}
Let $n\geq 0$. Then there exists an inverse system of abelian groups $\mathbf G_n$ such that $\cof(\mathbf G_n)=\aleph_n$ and ${\lim}^{n+1}\mathbf G_n\neq 0$.
\end{theorem}

\begin{proof}
Once again, our proof is by induction. If $n\geq 0$ and $\alpha<\lambda$ is an ordinal, let
\[
G^\lambda_{\alpha,n}= \bigoplus_{(\lambda\setminus\alpha)^{(n)}}\zet_2.
\]
We can think of $G^\lambda_{\alpha, n}$ as the Boolean group of finite subsets of $(\lambda \setminus \alpha)^{(n)}$, where the group operation is the symmetric difference.
For $\alpha < \beta< \lambda$, we let $p_\alpha^\beta \colon G^\lambda_{\beta, n}\to G^\lambda_{\alpha,n}$ be the inclusion map. Henceforth we omit the reference to these maps. We let 
\[
\mathbf G^\lambda_n = \{ G^\lambda_{\alpha, n}; p_\alpha^\beta \colon \alpha < \beta < \lambda \}.
\]
If $\lambda= \omega_n$, we write $G_{\alpha, n}$ for $G^\lambda_{\alpha, n}$ and $\mathbf G_n$ for $\mathbf G^\lambda_n$. Notice that $\cof(\mathbf G_n^\lambda)=\cof(\lambda)$. We plan to show that $\lim^{n+1} \mathbf G^\lambda_n \neq 0$, for every $\lambda$ of cofinality $\aleph_n$. %In particular, $\lim^{n+1}\mathbf G_n \neq 0$, for all $n$. 

Consider first the case $n=0$, and let $\lambda$ be an ordinal of cofinality $\omega$. Fix an increasing sequence $(\alpha_i)_{i\in \omega}$ converging to $\lambda$. Let $\bar x = (x_{\alpha,\beta})\in (\mathbf G^\lambda_0)^{(1)}$ be defined as follows:
\[
x_{\alpha, \beta}(\xi)=
\begin{cases}
1 \text { if } \alpha \leq \xi < \beta \text{ and } \xi = \alpha_i, \text{ for some } i \\
0 \text{ otherwise.}
\end{cases}
\] 
\begin{claim}
The sequence $\bar x=(x_{\alpha, \beta})\in (\mathbf G^\lambda_0)^{(1)}$ is coherent and nontrivial.
\end{claim}
\begin{proof}
Let $(\alpha,\beta,\gamma) \in \lambda^{(2)}$. We want to check that $\delta^2(\bar x)_{\alpha,\beta,\gamma}=\bar0$. Since $x_{\alpha,\alpha}=0$, for all $\alpha$, we may assume $\alpha < \beta < \gamma$. If $i\in\omega$ is such that $\alpha \leq \alpha_i <\beta$ then $x_{\alpha,\beta}(\alpha_i)=x_{\alpha,\gamma}(\alpha_i)=1$ and $x_{\beta,\gamma}(\alpha_i)=0$. Similarly, if $\beta \leq \alpha_i < \gamma$ then $x_{\alpha,\gamma}(\alpha_i)=x_{\beta,\gamma}(\alpha_i)=1$ and $x_{\alpha, \beta}(\alpha_i)=0$. Lastly, if $\xi\neq \alpha_i$, for some $i$ such that $\alpha \leq \alpha_i < \gamma$, then $x_{\alpha,\beta}(\xi)=x_{\alpha,\gamma}(\xi)=x_{\beta,\gamma}(\xi)=0$. Hence $\delta^2(\bar x)_{\alpha,\beta,\gamma}(\xi)=0$, for all $\xi < \lambda$. 

Suppose now that $\bar y=(y_\alpha)\in (\mathbf G^\lambda_0)^{(0)}$ trivializes $\bar x$. Since $y_0\in\bigoplus_{\xi <\lambda} \zet_2$ has finite support, there is $m \in \omega$ such that $y_0(\alpha_n)=0$, for all $n\geq m$. Since $x_{0, \alpha_m +1}= y_0 + y_{\alpha_m+1}$, we have that
\[
1=x_{0,\alpha_m+1}(\alpha_m)=y_0(\alpha_m)+y_{\alpha_m+1}(\alpha_m).
\]
Since $\alpha_m<\alpha_m+1$ and the support of $y_{\alpha_m+1}$ is contained in $[\alpha_m+1,\lambda)$, $y_{\alpha_m+1}(\alpha_m)=0$. Since $y_0(\alpha_m)=0$, we get $1=0$, a contradiction.
\end{proof}

Suppose $n >0$ and we have shown that $\lim^{n} \mathbf G^\lambda_{n-1}\neq 0$, for all $\lambda$ of cofinality $\aleph_{n-1}$. We have to show that $\lim^{n+1} \mathbf G^\lambda_{n} \neq 0$, for every $\lambda$ of cofinality $\aleph_{n}$. We only do it for $\lambda = \omega_{n}$, i.e., for the system $\mathbf G_{n}$, as the general case is only notationally more involved. 

We construct by induction a coherent sequence $\bar x\in \mathbf G^{(n+1)}_{n}$ with the additional property that if $\bar\alpha=(\alpha_0,\ldots,\alpha_{n+1})\in \omega_n^{(n+1)}$ then $ x_{\bar \alpha} \in G^{\alpha_{n+1}+1}_{\alpha_0, n}$. We start by letting $x_{\bar 0}=0$. Suppose we are at stage $\lambda < \omega_n$ and we have built $\bar x \restriction \lambda^{(n+1)}$. We can consider it as a coherent sequence in $(\mathbf G^\lambda_{n})^{(n+1)}$. It suffices to describe $\bar x \restriction (\lambda+1)^{(n+1)}$. In order words we have to define $x_{\bar \alpha}$ for $n+2$-tuples $\bar \alpha = (\alpha_0, \ldots, \alpha_{n+1})$ such that $\alpha_{n+1}= \lambda$, and we have to make sure that the extended sequence $\bar x \restriction (\lambda+1)^{(n+1)}$ is still coherent. Since $\cof(\lambda)\leq\aleph_{n-1}$, by Theorem~\ref{thm:goblot}, there exists $\bar y^\lambda \in (\mathbf G^\lambda_{n-1})^{(n)}$ that trivializes $\bar x \restriction \lambda^{(n+1)}$. 

\medskip

\noindent {\bf Case 1.} Suppose $\cof(\lambda) < \aleph_{n-1}$. For $\bar\alpha\in \lambda^{(n)}$, we set $x_{\bar\alpha,\lambda}=y^\lambda_{\bar\alpha}$. Then the sequence $\bar x \restriction (\lambda+1)^{(n+1)}$ defined in this way is still coherent.

\medskip

\noindent {\bf Case 2.} Suppose $\cof(\lambda)=\aleph_{n-1}$. Fix a coherent nontrivial sequence $\bar z^\lambda \in (\mathbf G^\lambda_{n-1})^{(n)}$. Such a sequence exists by the inductive assumption. We define a new coherent sequence $\bar w^\lambda$ as follows. For $\bar\alpha= (\alpha_0, \ldots, \alpha_n)\in \lambda^{(n)}$, let
\[
w^\lambda_{\bar\alpha}(\lambda_0,\ldots, \lambda_n)=
\begin{cases}
0 & \text { if } \lambda_n\neq\lambda\\
z^\lambda_{\bar\alpha}(\lambda_0,\ldots,\lambda_{n-1}) & \text{ if } \lambda_n=\lambda.
\end{cases}
\]
Note $w^\lambda_{\bar \alpha} \in G^{\lambda+1}_{\alpha_0,n}$ and that $\bar w^\lambda_{\bar \alpha}$ is essentially the same as $\bar z^\lambda_{\bar \alpha}$, but we add a new coordinate $\lambda$ to each $n$-tuple that occurs in the support of $z^\lambda_{\bar \alpha}$. Finally, for all $\bar \alpha \in \lambda^{(n)}$ we let:
\[
x_{\bar\alpha,\lambda}=y^\lambda_{\bar\alpha}+w^\lambda_{\bar\alpha}.
\]
Let us explain what we have done. First, we can always extend $\bar x \restriction \lambda^{(n+1)}$ to a coherent sequence $\bar x \restriction (\lambda+1)^{(n+1)}$ by defining $x_{\bar \alpha,\lambda}$ to be $y^\lambda_{\bar \alpha}$, for all $\bar \alpha \in \lambda^{(n)}$, but we are twisting the sequence $\bar y^\lambda$ by adding to it the coherent nontrivial sequence $\bar w^\lambda$. The important point is that $w^\lambda_{\bar \alpha}$ lives on the coordinates that do not appear in the support of $y^\lambda_{\bar \alpha}$. It is straightforward to check that the sequence $\bar x \restriction (\lambda +1)^{(n+1)}$ defined in this way is still coherent. 

\begin{claim}
The sequence $\bar x \in \mathbf G_n^{(n+1)}$ is nontrivial.
\end{claim}
\begin{proof} 
Suppose that $\bar u\in \mathbf G_n^{(n)}$ is such that $\delta^{n+1}(\bar u)=\bar x$. Note that, for every $\bar\alpha\in \omega_n^{(n)}$, the support of $u_{\bar \alpha}$ is a finite subset of $(\omega_n \setminus \alpha_0)^{(n)}$. 
Therefore, the set
 \[
C= \{\lambda < \omega_n \colon \supp(u_{\bar\alpha})\subseteq \lambda^{(n)}\text{ for all }\bar\alpha\in \lambda^{(n)}\}
 \]
 is a club. Hence we can find $\lambda \in C$ with $\cof(\lambda)=\aleph_{n-1}$. We want to show that the coherent sequence $\bar z^\lambda$ used in the construction of $\bar x \restriction (\lambda+1)^{(n+1)}$ is trivial, which is a contradiction. Let us first define an operation $d_\lambda$. Given a sequence $\bar v \in (\mathbf G_n^{\lambda+1})^{(n+1)}$ we let $d_\lambda(\bar v) \in (\mathbf G_{n-1}^{\lambda})^{(n)}$ be defined by letting, for $\bar \alpha \in \lambda^{(n-1)}$ and $(\lambda_0, \ldots, \lambda_{n-1}) \in (\lambda \setminus \alpha_0)^{(n-1)}$,
 \[
 d_\lambda(\bar v)_{\bar \alpha}(\lambda_0,\ldots, \lambda_{n-1}) = v_{\bar \alpha, \lambda}(\lambda_0, \ldots, \lambda_{n-1}, \lambda). 
 \]
 \noindent Note that $d_\lambda(\bar x \restriction (\lambda +1)^{(n+1)}) = \bar z^\lambda$. Now, recall that we are working with Boolean groups. So for $\bar v \in (\mathbf G_n^{\lambda +1})^{(n)}$ and $\bar \alpha \in \lambda^{(n)}$ we have: 
\[
\delta^{n+1}(\bar v)_{\bar \alpha, \lambda} = \sum_{0 \leq i \leq n} v_{\bar \alpha^i, \lambda} + v_{\bar \alpha}.
\]
Therefore, for all $(\lambda_0, \ldots, \lambda_{n-1}) \in (\lambda \setminus \alpha_0)^{(n-1)}$ we have: 
% $\bar \alpha \in \lambda^{(n)}$, and $(\lambda_0, \ldots, \lambda_{n-1}) \in (\lambda \setminus \alpha_0)^{(n-1)}$, we have: 
 \[
 d_\lambda(\delta^{n+1}(\bar v))_{\bar \alpha}(\lambda_0,\ldots, \lambda_{n-1})= 
 \delta^{n}(d_\lambda(\bar v))_{\bar \alpha}(\lambda_0,\ldots, \lambda_{n-1}) + v_{\bar \alpha}(\lambda_0,\ldots, \lambda_{n-1}, \lambda).
 \] 

%\medskip
\noindent Consider the sequence $\bar u$. For $\bar \alpha \in \lambda^{(n)}$, we have $\supp(u_{\bar \alpha}) \subseteq \lambda^{(n)}$, hence $u_{\bar \alpha}(\lambda_0,\ldots,\lambda_{n-1},\lambda)=0$, for all $(\lambda_0,\ldots, \lambda_{n-1})\in (\lambda \setminus \alpha_0)^{(n-1)}$. We know that $\delta^{n+1}(\bar u)= \bar x$, and we also know that $d_\lambda(\bar x \restriction (\lambda +1)^{(n+1)}) = \bar z^\lambda$. Putting this together we conclude that: 
\[
\delta^{n}(d_\lambda (\bar u \restriction (\lambda +1)^{(n)}))= \bar z^\lambda.
\]
%\medskip
 \noindent This contradicts the fact that $\bar z^\lambda$ is nontrivial and completes the proof. 
 \end{proof}
 Since $\bar x$ is a coherent nontrivial element of $\mathbf G_n^{(n+1)}$, the proof is complete.
\end{proof}

Let $\Lambda$ and $M$ be directed set and let $\phi\colon M \to \Lambda$ be an increasing function. Then, to every $\Lambda$-inverse system of abelian groups $\mathbf G = \{ G_\lambda; p^{\lambda'}_\lambda ; \lambda \leq \lambda ' \in \Lambda\}$ we can associate a $M$-inverse system $\phi^*(\mathbf G)= \mathbf H$. By definition, we let $\mathbf H = \{ H_\mu ; q^{\mu '}_\mu \colon \mu \leq \mu' \in M\}$, where $H_\mu = G_{\phi (\mu)}$, and $q^{\mu '}_\mu = p^{\phi(\mu')}_{\phi(\mu)}$. 

A function $\phi \colon M \to \Lambda$ is called {\em cofinal} if $\phi (M)$ is cofinal in $\Lambda$, i.e., for every $\lambda \in \Lambda$ there is $\mu \in M$ such that $\lambda \leq \phi (\mu)$. The following {\em cofinality theorem} was proved by B. Mitchell \cite{Mitchell73}, see also \cite[Theorem 14.9]{Mardesic.Book} for a proof.

\begin{theorem}[Mitchell, \cite{Mitchell73}]\label{thm:cofinal}
Let $\mathbf G$ be an inverse system indexed by a directed set $\Lambda$. Suppose $M$ is another directed set and $\phi\colon M \to \Lambda$ is an increasing cofinal map. Then, for every $n \geq 0$,
\[
\pushQED{\qed} 
{\lim}^n \mathbf G \cong {\lim}^n \phi^*(\mathbf G).\qedhere
\popQED
\]
\end{theorem} 

In particular, this implies that if $\mathbf G$ is an inverse system of abelian groups indexed by a directed set $\Lambda$, and $C$ is a cofinal subset of $\Lambda$, then $\lim^n \mathbf G \cong \lim^n \mathbf G \restriction C$, for all $n\geq 0$. Moreover, note that if $\Lambda$ is a directed set of cardinality and cofinality $\aleph_n$ then there is an increasing cofinal map $\phi\colon \Lambda \to \omega_n$. Therefore, from Theorem \ref{thm:mitchell} and Theorem \ref{thm:cofinal}, we get the following corollary. 

\begin{corollary}\label{general-cofinality-thm} Let $n\geq 0$ and let $\Lambda$ be a directed set of cofinality $\aleph_n$. Then there is an inverse system of abelian groups $\mathbf G$ indexed by $\Lambda$ such that $\lim^{n+1} \mathbf G \neq 0$. \qed
\end{corollary} 

We only prove an easy instance of Theorem \ref{thm:cofinal} which is sufficient for our purposes. 

\begin{proposition}\label{prop:cofinallinear}
Let $\mathbf G$ be an inverse system of abelian groups indexed by a directed set $\Lambda$. Suppose that $C\subseteq\Lambda$ is a linear cofinal subset, and that $\bar x\in\mathbf G\restriction C^{(n)}$ is coherent. Then there is a coherent $\bar z\in\mathbf G^{(n)}$ such that $\bar z\restriction C^{(n)}=\bar x$.
\end{proposition}

\begin{proof}
Fix $\Lambda$, $\mathbf G$, $C$, and $n$ as in the hypothesis. Note that, by replacing $C$ by a cofinal subset if necessary, we may assume that $C$ is well-ordered. For $\lambda \in \Lambda$, let $s(\lambda)$ be the least $\mu \in C$ such that $\lambda \leq \mu$. Note that $s$ is increasing and equals the identity on $C$. Suppose $\bar x \in \mathbf G\restriction C^{(n)}$ is coherent. We define its extension $\bar z \in \mathbf G ^{(n)}$. Given $(\lambda_0, \ldots, \lambda_n) \in \Lambda^{(n)}$, we let: 
 \[ 
 z_{\lambda_0, \ldots, \lambda_n} = p^{s(\lambda_0)}_{\lambda_0}(x_{s(\lambda_0), \ldots, s(\lambda_n)}). 
 \]
 Note that $\bar z$ is also coherent. Since $s$ is the identity on $C$, we get that $\bar z\restriction C^{(n)}= \bar x$, as required. 
 \end{proof}

\section{The inverse system $\mathbf A$}\label{S.A}

In this section we consider inverse systems of a special form. We fix an infinite set $X$ and for each $x\in X$, a nontrivial abelian group $G_x$. We also fix $\mathcal I$, a nonprincipal ideal on $X$ containing the Fr\'echet ideal, i.e., all finite subsets of $X$. For $a \in \mathcal I$ we consider the group: 
\[
G_a = \bigoplus_{x \in a} G_x.
\]
If $a\subseteq b$ we let $p^b_a$ be the natural projection from $G_b$ to $G_a$. Let 
\[
\mathbf G_{\mathcal I} = \{ G_a, p^b_a\colon a,b \in \mathcal I, a\subseteq b \}.
\]
Note that the map $p^b_a$ is surjective, whenever $a\subseteq b$, and hence $\lim^0 \mathbf G_{\mathcal I}\neq 0$. We are interested in computing $\lim^n \mathbf G_{\mathcal I}$ for $n\geq 1$. We focus on the case $G_x = \mathbb Z$ or $G_x = \mathbb Z_2$, for all $x\in X$. The special case where $X=\omega \times \omega$, the ideal $\mathcal I_0$ is generated by sets 
\[
a_f= \{ (i,j) \in \omega \times \omega \colon j \leq f(i) \},
\]

%\medskip

\noindent for functions $f\in \omega^\omega$, and $G_{(i,j)}=\mathbb Z$, for all $(i,j) \in \omega^2$, corresponds to the inverse system $\mathbf A$ studied in \cite{MardesicPrasolov}, \cite{Dow-Simon-Vaughan}, \cite{Bergfalk2017} \cite{BergfalkLH} and \cite{Bergfalk2021simutaneously}, and is the main object of our study, but we find it convenient to consider a slightly more general situation. Our goal is to isolate a combination of combinatorial principles which together imply that $\lim^n \mathbf A \neq 0$. 

In order to understand the systems $\mathbf G_{\mathcal I}$ as above it is convenient to define an associated inverse system $\mathbf P_{\mathcal I}$. For $a\in \mathcal I$ we let: 
\[
P_a = \prod_{x \in a} G_x.
\]

For $a \subseteq b$ we denote also by $p^b_a$ the natural projection map from $P_b$ to $P_a$. Note that $G_a$ is a (normal) subgroup of $P_a$ and hence we can construct the quotient inverse system $\mathbf P/\mathbf G_{\mathcal I}$, by letting, for all $a\in \mathcal I$,
\[
(P/G)_a = P_a/G_a.
\]
We then have a short exact sequence of $\mathcal I$-inverse systems of abelian groups: 
\[
\mathbf 0\to\mathbf G_{\mathcal I}\to\mathbf P_{\mathcal I}\to\mathbf P/\mathbf G_{\mathcal I} \to \mathbf 0.
\]

%\medskip

Since $\lim$ is a left exact functor, by a standard construction in homological algebra (see e.g., \cite[Chapter IV]{Hilton-Stammbach}), this gives rise to a long exact sequence:

\begin{equation}\label{eq:long-exact}
\cdots\to{\lim}^n\mathbf G_{\mathcal I}\to{\lim}^n\mathbf P_{\mathcal I}\to{\lim}^n \mathbf P/\mathbf G_{\mathcal I}
\to {\lim}^{n+1}\mathbf G_{\mathcal I}\to\cdots.
\end{equation}
%\medskip
\begin{lemma}
If $n \geq 1$, then $\lim^n \mathbf P_{\mathcal I}=0$.
\end{lemma} 
\begin{proof}
 Let $\bar x \in \mathbf P_{\mathcal I}^{(n)}$ be a coherent sequence, i.e., such that $\delta^{n+1}(\bar x)=0$. We have to find $\bar y \in \mathbf P_{\mathcal I}^{(n-1)}$ such that $\delta^n(\bar y) = \bar x$. Suppose $\bar a = (a_0, \ldots, a_{n-1})\in \mathcal I^{(n)}$. We define $y_{\bar{a}}\in G_{a_0}$. Suppose $k \in a_0$. If $a_0 = \{ k\}$ we let $y_{\bar a}(k) =0$. Otherwise, we let 
\[
y_{\bar a}(k) = x_{\{ k \}, a_0,\ldots, a_{n-1}}(k).
\]
It is straightforward to check that $\delta^n(\bar y) =\bar x$ using the coherence of $\bar x$. 
\end{proof}
Let us also define explicitly the map 
\[
\varphi_n\colon {\lim}^{n} \mathbf P/\mathbf G_{\mathcal I} \to {\lim}^{n+1}\mathbf G_{\mathcal I}.
\]
Let $\bar x\in\mathbf P/\mathbf G^{(n)}_{\mathcal I}$ be such that $\delta^{n+1}(\bar x)=0$. Note that ${\bar x} = (x_{\bar a}\colon {\bar a} \in \mathcal I^{(n)})$ where $x_{\bar a} \in P_{a_0}/G_{a_0}$. Let $y_{\bar a} \in P_{a_0}$ be such that $x_{\bar a} = G_{a_0} + y_{\bar a}$. Let $\bar y = (y_{\bar a}\colon \bar a \in \mathcal I^{(n)})$. Since $\delta^{n+2}\circ \delta^{n+1} =0$, we have that $\delta^{n+1}(\bar y) \in {\rm ker}(\delta^{n+2})$. We then let $\varphi_n(\bar x)= \delta^{n+1}(\mathbf G_{\mathcal I}^{(n)})+ \delta^{n+1}(\bar y)$. The fact that $\varphi_n(\bar x)$ is coherent follows from the fact that $\delta^{n+2}\circ \delta^{n+1} =0$. To see that the definition of $\varphi_n(\bar x)$ does not depend on the choice of $\bar y$, let $\bar y'$ be another such sequence and note that $\bar y' - \bar y \in \mathbf G_{\mathcal I}^{(n)}$, hence $\delta^{n+1}(\bar y' - \bar y) \in \delta^{n+1}(\mathbf G_{\mathcal I}^{(n)})$. The fact that $\varphi_n$ is a homomorphism is straightforward. Let us check that it is injective. Suppose $\bar x\in \mathbf P/\mathbf G_{\mathcal I}^{(n)}$ is such that $\varphi_n(\bar x)=0$. Let $\bar y$ be a lifting of $\bar x$, i.e., $y_{\bar a}\in P_{a_0}$ is such that $x_{\bar a} = G_{a_0} + y_{\bar a}$, for all $\bar a \in \mathcal I^{(n)}$. Then $\delta^{n+1}(\bar y) \in \delta^{n+1}(\mathbf G_{\mathcal I}^{(n)})$. Fix $\bar z \in \mathbf G_{\mathcal I}^{(n)}$ such that $\delta^{n+1}(\bar y)= \delta^{n+1}(\bar z)$. Therefore $\delta^{n+1}(\bar y - \bar z)=0$. Since $\lim^n \mathbf P_{\mathcal I} =0$ there is $\bar u \in \mathbf P_{\mathcal I}^{(n)}$ such that $\delta^n(\bar u) =\bar y - \bar z$. Let $\bar v \in \mathbf P/ \mathbf G_{\mathcal I}^{(n-1)}$ be defined by $v_{\bar a}= G_{a_0}+ u_{\bar a}$, for all $\bar a$. Then $\delta^n (\bar v)= \bar x$, as desired. Finally, let us check that $\varphi_n$ is onto. Let $\bar y \in \mathbf G_{\mathcal I}^{(n+1)}$ be such that $\delta^{n+2}(\bar y)= 0$. We can view $\bar y$ as a member of $\mathbf P_{\mathcal I}^{(n+1)}$. Since $\lim^{n+1} \mathbf P_{\mathcal I}=0$ we can find $\bar z \in \mathbf P_{\mathcal I}^{(n)}$ such that $\delta^{n+1}(\bar z)= \bar y$. Now let $\bar u$ be defined by $u_{\bar a} = G_{a_0}+ z_{\bar a}$, for all $\bar a$. It follows that $\delta^{n+1}(\bar u)=0$ and $\varphi_n(\bar u)= \bar y$, as required. 

\medskip

The upshot of this analysis is that $\lim^1 \mathbf G_{\mathcal I}$ is isomorphic to the quotient of $\lim^0 \mathbf P/\mathbf G_{\mathcal I}$ by $\lim^0 \mathbf P$, and $\lim^{n+1} \mathbf G_{\mathcal I}$ is isomorphic to $\lim^n \mathbf P/ \mathbf G_{\mathcal I}$, for all $n\geq 1$.

\medskip

We now consider the question whether $\lim^n \mathbf G_{\mathcal I}$ vanishes or not depending on the cofinality of $\mathcal I$ under inclusion. First note that if ${\rm cof}(\mathcal I)= \aleph_0$, then by Theorem \ref{thm:flasquegoblot} $\lim^n \mathbf G_{\mathcal I}=0$, for all $n \geq 1$. In order to show that $\lim^{n+1} \mathbf G_{\mathcal I}\neq 0$, by \eqref{eq:long-exact}, it suffices to show that $\lim^n \mathbf P/\mathbf G_{\mathcal I} \neq 0$. We find it convenient to replace $\mathbf P/\mathbf G_{\mathcal I}$ by a slightly different system $\mathbf Q_{\mathcal I}$. Consider the ordered set $(\mathcal I,\subseteq_*)$, where $\subseteq_*$ denotes inclusion modulo finite sets. Let
\[
\mathbf Q_{\mathcal I}= \{ P_a/G_a , p^b_a \colon a\subseteq_* b \in \mathcal I\}.
\]
Since $\mathcal I$ contains the Fr\'echet ideal, any coherent nontrivial sequence $\bar x \in \mathbf Q_{\mathcal I}^{(n)}$ is also coherent and nontrivial in $\mathbf P\mathbf /\mathbf G_{\mathcal I}$. 
% is also the identity map ${\rm id}_{\mathcal I}$ is increasing and cofinal from $(\mathcal I, \subseteq)$
%to $(\mathcal I, \subseteq_*)$. Then, following the notation introduced before Theorem \ref{thm:cofinal}, 
%${\rm id}_{\mathcal I}^*(\mathbf Q_{\mathcal I})= \mathbf P/\mathbf G_{\mathcal I}$.
%Therefore, by Theorem \ref{thm:cofinal}, we have that 
%\[
%{\lim}^{n} \mathbf Q_{\mathcal I} \cong {\lim}^{n} \mathbf P/\mathbf G_{\mathcal I}. 
%\] 
%We find it more to work with the system $\mathbf Q_{\mathcal I}$. instead of $\mathbf P/\mathbf G_{\mathcal I}$.
Henceforth, we will work with $\mathbf Q_{\mathcal I}$ instead of $\mathbf P/\mathbf G_{\mathcal I}$. To simplify notation, when we are working with this system we will write $Q_a$ for $P_a/G_a$, when $a\in \mathcal I$. 

Let us now turn to the first interesting case and assume that $\mathcal I$ is an ideal on a countable set $X$, and ${\rm cof}(\mathcal I)=\aleph_1$. We would like to know if $\lim^1 \mathbf G_{\mathcal I}$ vanishes or not. It turns out that the answer depends on the choice of groups $G_x$, for $x\in X$. Recall that two functions with the same domain $f$ and $g$ are said almost equal, written $f=_*g$, if the set of points on which they differ is finite.

\begin{proposition}\label{ideals-Z} Suppose $X$ is countable and $G_x= \mathbb Z$, for all $x\in X$. 
Let $\mathcal I$ be an ideal on $X$ with ${\rm cof}(\mathcal I)= \aleph_1$. Then $\lim^1 \mathbf G_{\mathcal I} \neq 0$. 
\end{proposition} 

\begin{proof} 
We may assume $X= \omega$. We show that there is a coherent nontrivial sequence in $\mathbf Q_{\mathcal I}^{(0)}$. Let $\{ a_\xi \colon \xi < \omega_1 \}$ be a cofinal subset of $\mathcal I$ under $\subseteq_*$. We may assume that all the $a_\xi$ are infinite, and no $a_\xi$ belongs to the ideal generated by finite sets and the $a_\eta$, for $\eta < \xi$. It suffices to show that there is a family $\{ f_\xi \colon \xi < \omega_1 \}$ of functions such that: 
\begin{itemize} 
\item $f_\xi \colon a_\xi \to \mathbb Z$, for all $\xi<\omega_1$,
\item if $\xi \neq \eta$ then $f_\xi \! \restriction \! (a_\xi \cap a_\eta) =_* f_\eta \! \restriction \! (a_\xi \cap a_\eta)$, 
\item there is no $f\colon \omega \to \mathbb Z$ such that $f \! \restriction \! a_\xi =_* f_\xi$, for all $\xi<\omega_1$. 
\end{itemize} 

\noindent The second condition states that the family $\{ f_\xi \colon \xi < \omega_1\}$ is coherent and the third that it is nontrivial. To begin, pick by induction on $\xi$, an infinite subset $b_\xi$ of $a_\xi$ such that $b_\xi \cap a_\eta$ is finite for all $\eta < \xi$. For each $\xi$, let $e_\xi\colon b_\xi \to b_\xi$ be defined by:
\[
e_\xi (n)= \min (b_\xi \setminus (n+1)),
\]
i.e., if $n\in b_\xi$ then $e_\xi(n)$ is the next element of $b_\xi$ above $n$. Now, by induction on $\xi$, we build functions $f_\xi \colon a_\xi \to \mathbb Z$ such that:
\begin{itemize} 
\item if $\eta < \xi$ then $f_\eta \! \restriction \! (a_\eta \cap a_\xi ) =_* f_\xi \! \restriction \! (a_\eta \cap a_\xi)$,
\item $f_\xi \! \restriction \! b_\xi =_* e_\xi$, for all $\xi$. 
\end{itemize} 
Suppose we have built $f_\eta$, for $\eta < \xi$. We need to construct $f_\xi$. First, by an easy diagonalisation argument we can find integers $n_\eta$, for $\eta <\xi$, such that:
\begin{itemize} 
\item $a_\eta \cap b_\xi \subseteq n_\eta$, for all $\eta < \xi$, 
\item if $\eta < \zeta < \xi$ then $f_\eta (k) = f_\zeta (k)$, for all $k \in (a_\eta \setminus n_\eta) \cap (a_\zeta \setminus n_\zeta)$. 
\end{itemize} 
For $n \in a_\xi$ we define $f_\xi(n)$ as follows: 
\[
f_\xi(n)=
\begin{cases}
f_\eta (n) &\text{ if } n\in a_\eta \setminus n_\eta,\\
e_\xi (n) &\text{ if } n \in b_\xi ,\\
0 &\text{ otherwise.} 
\end{cases}
\]
It is clear that the $f_\xi$ constructed in this way are as desired. Let us now show that there is no $f\colon \omega \to \mathbb Z$ such that $f \! \restriction \! a_\xi =_* f_\xi$, for all $\xi$. Suppose otherwise and fix such $f$. Let $m_\xi$ be the least such that $f$ agrees with $e_\xi$, for all $n \in b_\xi \setminus m_\xi$. We can then find an uncountable $Z\subseteq\omega_1$ and an integer $m$ such that $m_\xi = m$, for all $\xi \in Z$. Since all the $b_\xi$ are infinite, we can find distinct $\eta$ and $\xi$ in $Z$ and $k\geq m$ such that $k \in b_\eta \cap b_\xi$. We can then define $k_0=k$, and $k_{i+1}= f(k_i)$, for all $i$. By the definition of $e_\eta$ and $e_\xi$ we have that: 
\[
b_\eta \setminus k = \{ k_0, k_1, \ldots, k_i, \ldots \} = b_\xi \setminus k,
\]
which contradicts the fact that $b_\eta$ and $b_\xi$ are almost disjoint. 
\end{proof}

We now turn to the case $X$ is countable, $G_x= \mathbb Z_2$, for all $x\in X$, and $\mathcal I$ as an ideal on $X$ of cofinality $\aleph_1$. It turns out that the question whether $\lim^1 \mathbf G_{\mathcal I}$ vanishes in this case depends on additional axioms of set theory. In fact, it was shown in \cite[Theorem 4.26]{shelah_2017} that it is relatively consistent with $\ZFC$ that there is an almost disjoint family $\mathcal A$ of size $\aleph_1$ of subsets of $\omega$ such that the ideal $\mathcal I$ generated by $\mathcal A$ and the Fr\'echet ideal has the uniformization property. In our terminology it says that if we let $G_n= \mathbb Z_2$, for all $n$, then $\lim^1 \mathbf G_{\mathcal I}=0$. This is complemented by the following. ($\MA_{\aleph_1}$ is Martin's Axiom, see \cite[III.3]{Kunen.2011})

\begin{proposition}\label{ideals-Z2} 
Suppose $X$ is countable and $G_x= \mathbb Z_2$, for all $x\in X$. Let $\mathcal I$ be an ideal on $X$ with ${\rm cof}(\mathcal I)= \aleph_1$. Suppose that either $2^{\aleph_0} < 2^{\aleph_1}$ or ${\rm MA}_{\aleph_1}$ holds. Then $\lim^1 \mathbf G_{\mathcal I} \neq 0$. 
\end{proposition} 

\begin{proof} 
We again assume $X=\omega$. Let $\{ a_\xi \colon \xi <\omega_1 \}$ be a cofinal subset of $\mathcal I$ under $\subseteq_*$ such that no $a_\xi$ is in the ideal generated by finite sets and the $a_\eta$, for $\eta <\xi$. As in the proof of Proposition \ref{ideals-Z} we construct a nontrivial coherent family of functions $\{f_\xi \colon \xi < \omega_1\}$, but this time the functions $f_\xi$ take values in $\mathbb Z_2$ rather than in $\mathbb Z$. Pick for each $\xi <\omega_1$, an infinite subset $b_\xi$ of $a_\xi$ which has finite intersection with the $a_\eta$, for $\eta <\xi$. Now, as in Proposition~\ref{ideals-Z}, given a family of functions $\{ g_\xi \colon \xi <\omega_1\}$ with $g_\xi\colon b_\xi \to \mathbb Z_2$, for $\xi <\omega_1$, we can extend as before each $g_\xi$ to a function $f_\xi$ on $a_\xi$ such that the family $\{f_\xi \colon \xi <\omega_1\}$ is coherent. So, it suffices to show that there is a choice of such functions $\{ g_\xi\colon \xi < \omega_1\}$ which is nontrivial, i.e., there is no $f\colon \omega \to \mathbb Z_2$ such that $f\restriction b_\xi =_* g_\xi$, for all $\xi <\omega_1$. 

Suppose first that $2^{\aleph_0} < 2^{\aleph_1}$. For every $h\colon \omega_1 \to \mathbb Z_2$ and $\xi <\omega_1$, we let we $g^h_\xi$ be the constant function taking value $h(\xi)$ on $b_\xi$. Note that any $f\colon \omega \to \mathbb Z_2$ can trivialize the family $\{ g^h_\xi\colon \xi <\omega_1\}$ for at most one $h\colon \omega_1 \to \mathbb Z_2$. Therefore, by simple counting there is $h\colon \omega_1 \to \mathbb Z_2$ such that the sequence $\{ g^h_\xi \colon \xi < \omega_1\}$ is nontrivial. 

Let us now assume $\MA_{\aleph_1}$. For each $\xi <\omega_1$, let as before $e_\xi$ be the function defined on $b_\xi$ such that if $k\in b_\xi$ then $e_\xi(k)$ is the next element of $b_\xi$ above $k$. We cannot use the function $e_\xi$ as in Proposition \ref{ideals-Z}, but by ${\rm MA}_{\aleph_1}$ we can find a $2$-branching subtree $T$ of $\omega^{<\omega}$ such that the set $Y$ of all $\xi <\omega_1$ such that $e_\xi \in [T]$ is uncountable. For each $s\in T$ let $u^s_0$ and $u^s_1$ be the two integers $u$ such that $s\, \, \widehat \, \, u \in T$. For $\xi \in Y$, let $g_\xi\colon b_\xi \to \mathbb Z_2$ be defined as follows. Suppose $k \in b_\xi$ and let $l= | b_\xi \cap (k+1)|$. Let
\[
g_\xi (k)= i \mbox{ if and only if } e_\xi(k)= u^{e_\xi \restriction l}_i. 
\] 
We claim that the sequence of functions $\{ g_\xi \colon \xi \in Y \}$ is nontrivial. Indeed, suppose $f\colon \omega \to \mathbb Z_2$ is a trivializing function. For each $\xi \in Y$ there is an integer $k_\xi \in b_\xi$ such that $g_\xi \restriction (b_\xi \setminus k_\xi) \subseteq f$. By the pigeon hole principle there is an uncountable $Z\subseteq Y$ and an integer $k$ such that $k_\xi = k$, for all $\xi \in Z$. By shrinking $Z$ if necessary we may assume that there is an integer $l$ and a sequence $s$ of length $l$ such that $| b_\xi \cap (k+1)|= l$ and $e_\xi \restriction l = s$, for all $\xi \in Z$. Let $\xi$ and $\eta$ be distinct elements of $Z$. % We show that $b_\xi = b_\eta$. 
We already know that $b_\xi \cap (k+1)= b_\eta \cap (k+1)=s$. Suppose $m \geq k$ is a common element of $b_\xi$ and $b_\eta$ and $b_\xi \cap (m +1)= b_\eta \cap (m+1)$. Let $t$ be the increasing enumeration of this finite set. We show that the next elements of $b_\xi$ and $b_\eta$ above $m$ are equal as well. Indeed, the next element of $b_\xi$ above $m$ is $u^t_{g_\xi(m)}$ and the next element of $b_\eta$ above $m$ is $u^t_{g_\eta(m)}$. Since $m\geq k$ we know that $g_\xi(m)$ and $g_\eta(m)$ are both equal to $f(m)$, hence the next element above $m$ of both $b_\xi$ and $b_\eta$ is equal to $u^t_{f(m)}$. It follows by induction that $b_\xi = b_\eta$, which is a contradiction. 
\end{proof} 
It turns out that no additional set theoretic hypothesis is needed for the above result if the ideal $\mathcal I$ is generated by a set linearly ordered under inclusion modulo the Fr\'echet ideal. 
%Set theory does not enter play in case the ideal of interest is generated by a linear set. 
Recall that a {\em tower} on $\omega$ is a sequence $\{ a_\xi \colon \xi <\kappa\}$ of subsets of $\omega$ such that $a_\xi \subset_* a_\eta$, for all $\xi < \eta <\kappa$. The following is an instance of a many faceted phenomenon which can be traced to Hausdorff gaps. It can be obtained for example from \cite[Proposition 4.22]{Bekkali1991}, \cite[Theorem 2.4]{Dow-Simon-Vaughan}, and can also be proved using Todor\v{c}evi\'c's $\rho$-functions and $C$-sequences (see \cite{Todorcevic2007}). We include a slightly different proof for completeness. 

\begin{proposition}\label{thm:hausgap}
Suppose $\mathcal I$ is an ideal on $\omega$ generated by a tower of length $\omega_1$. Let $G_n= \mathbb Z_2$, for all $n$. Then $\lim^1 \mathbf G_{\mathcal I} \neq 0$. 
\end{proposition}

We will use the following lemma.

\begin{lemma}\label{L.orderings}
There are orderings $<_n$ on $\omega_1$, for $n< \omega$, such that: 
\begin{enumerate}
\item\label{enum1} $(\omega_1,<_n)$ is a tree ordering of height at most $\omega$, for all $n$, 
\item\label{enum2} if $\alpha <_m \beta$ and $m <n$, then $\alpha <_n \beta$, 
\item\label{enum3} if $\alpha < \beta < \omega_1$ then there is $n$ such that $\alpha <_n \beta$, 
\item\label{enum4} if $\alpha <_n \beta$, then $a_\alpha \setminus n \subseteq a_\beta$,
\item\label{enum5} if $\alpha$ is a limit and $\alpha<_n\beta$ then $\alpha+1\leq_n\beta$.
\end{enumerate}
\end{lemma}
\begin{proof}
We construct the orderings $<_n \restriction \! \! \gamma$, for all $n$, by induction on $\gamma$. Let $\gamma$ be a countable ordinal and suppose that the orders $<_n \restriction \! \gamma$, have been constructed so that they satisfy conditions \eqref{enum1}--\eqref{enum5} with $\gamma$ instead of $\omega_1$. We need to extend each $<_n \restriction \gamma$ to $<_n\restriction \gamma +1$, while preserving the conditions \eqref{enum1}--\eqref{enum5} this time with $\gamma +1$ instead of $\omega_1$. This means we have to decide where to put $\gamma$ in each of the orderings $<_n$. Assume first that $\gamma$ is a successor, e.g., $\gamma = \alpha +1$, for some $\alpha$. Let $n$ be the least such that $a_\alpha \setminus n \subseteq a_\gamma$. If $k < n$ we put $\gamma$ on the $0$-th level of $<_k$, i.e., we make $\gamma$ incomparable with all $\beta <\gamma$ in the sense of $<_k$. If $k \geq n$ we put $\gamma$ on top of $\alpha$, i.e., we let $\beta <_k \gamma$ iff $\beta \leq_k \alpha$. Notice that conditions \eqref{enum1}--\eqref{enum5} with $\omega_1$ replaced by $\gamma+1$ are preserved.

Suppose now that $\gamma$ is a limit ordinal. Since $\gamma$ has countable cofinality, we can fix an increasing sequence of successor ordinals $(\gamma_k)_k$ converging to $\gamma$. Define $n_k$ as the least $n>n_{k-1},k$ such that
\begin{itemize}
\item $\gamma_0 <_n \gamma_1 <_n \cdots <_n\gamma_k$,
\item $a_{\gamma_i}\setminus n\subseteq a_\gamma$, for each $i \leq k$.
\end{itemize}
For $i <n_0$ we let $\gamma$ be incomparable in $<_i$ with all $\beta <\gamma$. If $i\in [n_k,n_{k+1})$ we put $\gamma$ as an immediate successor of $\gamma_k$ in $<_i$, that is, for $\beta <\gamma$ we let $\beta <_i \gamma$ iff $\beta \leq_i \gamma_k$. This ends the definition of the orderings $<_n \restriction \gamma +1$. We need to verify that all the conditions are verified, but this is a trivial observation from the choice of $n_k$. Note that condition \eqref{enum5} is ensured by the fact that all $\gamma_k$ used in the above construction are successor ordinals. 
\end{proof}

\begin{proof}[Proof of Proposition~\ref{thm:hausgap}]
Let $\{ a_\xi \colon \xi < \omega_1\}$ be a tower on $\omega$ generating the ideal $\mathcal I$ modulo the Fr\'echet ideal. We construct a coherent nontrivial family of functions $\{ f_\alpha \colon \alpha < \omega_1\}$ such that $f_\alpha \colon a_\alpha \to \mathbb Z_2$, for all $\alpha<\omega_1$. Given Lemma~\ref{L.orderings}, we construct $b_\alpha\subseteq a_\alpha$, for all $\alpha < \omega_1$, such that
\begin{itemize} 
\item if $\alpha < \beta$ then $b_\beta \cap a_\alpha =_* b_\alpha$, 
\item there is no $b\subseteq \omega$ such that $b \cap a_\alpha =_* b_\alpha$, for all $\alpha < \omega_1$. 
\end{itemize} 
Suppose that such sets are constructed. Let $f_\alpha\colon a_\alpha \to \mathbb Z_2$, for $\alpha < \omega_1$, be defined by 
\[
f_\alpha(n)=
\begin{cases}1&\text{ if }n\in b_\alpha\\
0&\text{ otherwise}
\end{cases}
\]
Then $\{f_\alpha \colon \alpha < \omega_1\}$ is the required coherent nontrivial family of functions. 

Fix $\alpha$ and let $\xi<\alpha$. Let $n$ the least such that $\xi<_n\alpha$, and $\eta$ be the least ordinal $>\xi$ such that $\xi <_n \eta \leq_n \alpha$. Define $n_{\xi,\alpha}$ be the least integer in $a_\eta\setminus a_\xi$ which is bigger than $n$. Let
\[
b_\alpha=\{ n_{\xi,\alpha} \mid\xi < \alpha \}.
\]
We need to verify that the $b_\alpha$ thus constructed have the required properties. Suppose first that $\alpha<\beta$. We want to show that $b_\beta\cap a_\alpha=_*b_\alpha$. Let $m$ be the least such that $\alpha<_m\beta$. Since the height of $(\omega_1,<_{m-1})$ is at most $\omega$, the set $F_{m-1}(\beta)= \{ \xi <\beta \colon \xi <_{m-1} \beta\}$ is finite. Suppose $\xi \in \beta \setminus F_{m-1}(\beta)$ and let $n$ be the least such that $\xi <_n \beta$. Then $n \geq m$ and since $\alpha <_m \beta$, the ordering $<_n$ is a tree ordering and extends $<_m$, we have that $\xi$ and $\alpha$ are comparable in $<_n$. Moreover, if $\xi < \alpha$ then $n_{\xi,\alpha} = n_{\xi, \beta}$, and if $\xi \geq \alpha$ then $n_{\xi, \alpha} \notin a_\alpha$. This implies that $b_\beta \cap a_\alpha =_* b_\alpha$. 

Let us now show that there is no set $b$ such that $b\cap a_\alpha=_* b_\alpha$ for all $\alpha$. First, note that if $\alpha < \beta$ then $b_{\beta} \cap (a_{\alpha+1} \setminus a_\alpha)$ is finite. The reason is that for some $k$ we will have $\alpha <_n \alpha+1 \leq_n \beta$, for all $n \geq k$. So, for all $\xi \in \beta \setminus F_k(\beta)$, if $\xi <\alpha$ then $n_{\xi, \beta}\in a_\alpha$, and if $\xi >\alpha$ then $n_{\xi, \beta} \notin a_{\alpha +1}$. Suppose now $b$ is such that $b \cap a_\alpha = _* b_\alpha$, for all $\alpha$. Then, for each $\alpha$, there is $m_\alpha$ such that $(b \cap a_\alpha) \setminus m_\alpha = b_\alpha\setminus m_\alpha$ and $b \cap (a_{\alpha+1} \setminus a_\alpha) \subseteq m_\alpha$. By increasing $m_\alpha$, we may require that $\alpha <_{m_\alpha} \alpha +1$. Now, there is a fixed $m$ such that $m_\alpha = m$, for an uncountable set $X$ of limit ordinals. Take some $\beta \in X$ such that $X \cap \beta$ is infinite. Then there is $\alpha \in X \cap \beta$ such that the least $n$ such that $\alpha <_n \beta$ is bigger than $m$. By \eqref{enum5} $\alpha +1 \leq_n \beta$, and so $n_{\alpha,\beta} \in a_{\alpha +1}\setminus a_\alpha$. It follows that $n_{\alpha,\beta} \in b_\beta \setminus m$, thus $n_{\alpha,\beta} \in b$. On the other hand we have that $b \cap (a_{\alpha+1} \setminus a_\alpha) \subseteq m$, so $n_{\alpha,\beta} <m$, a contradiction.
\end{proof}

Let us now turn to higher derived limits. We will consider the following statement, for a countable abelian group $G$. 
 
 \medskip

\begin{description}
\item[$\varphi_n(G)$] {\em Suppose $\mathcal I$ is an ideal on $\omega$ generated by a tower of length $\omega_n$. Set $G_k= G$, for all $k< \omega$, and let $\mathbf G_{\mathcal I}$ be the associated inverse system of abelian groups. Then $\lim^n \mathbf G_{\mathcal I} \neq 0$. }
 \end{description}
 
 \medskip
 
By Proposition \ref{thm:hausgap}, the statement $\varphi_1(\mathbb Z_2)$ holds in $\ZFC$. It follows that $\varphi_1(G)$ holds, for every nontrivial abelian group $G$. We intend to leverage the statement $\varphi_n(G)$ with a certain combinatorial principle to get $\varphi_{n+1}(G)$. Recall the following weak version of diamond. 
 
 \medskip

\begin{description}
\item[${\rm w}\diamondsuit (S)$] {\em Suppose $\kappa$ is an uncountable regular cardinal and $S$ is a stationary subset of $\kappa$. Then, for every $F\colon 2^{ < \kappa} \to 2$, there is $g\colon \kappa \to 2$ such that, for every $f\colon \kappa \to 2$, the set $\{ \alpha \in S \colon g(\alpha)= F( f \restriction \alpha) \}$ is stationary.}
\end{description}
 
 \medskip
 
This principle was introduced by Devlin and Shelah \cite{Devlin-Shelah}, who showed that ${\rm w}\diamondsuit(\omega_1)$ is equivalent to $2^{\aleph_0} < 2^{\aleph_1}$. More generally, ${\rm w}\diamondsuit(\kappa^+)$ is equivalent to $2^\kappa < 2^{\kappa^+}$, for every regular $\kappa$. For our purposes we will require that weak diamond holds on a specific stationary subset of $\omega_{n+1}$, namely on the set
\[
S^n_{n+1}= \{ \alpha < \omega_{n+1}\colon {\rm cof}(\alpha) = \omega_n \}.
\]
The following stepping-up lemma is tha main technical tool we will use. 

%We intend to show that if we have ${\rm w}\diamondsuit(S^n_{n+1})$ and $\varphi_n(G)$ holds, for some group $G$, 
 %then so does $\varphi_{n+1}(G)$. 

\begin{lemma}\label{stepping-up} 
Let $n\geq 1$ and let $G$ be a countable abelian group. Suppose ${\rm w}\diamondsuit(S^n_{n+1})$ and $\varphi_n(G)$ both hold, then so does $\varphi_{n+1}(G)$. 
 \end{lemma} 

\begin{proof} 
Let $\mathcal I$ be an ideal on $\omega$ generated by a tower $C=\{ a_\xi \colon \xi < \omega_{n+1}\}$. Let $\mathbf G_{\mathcal I}$ be the system obtained by setting $G_k=G$, for all $k\in\omega$. Let $\mathbf Q_{\mathcal I}$ be the associated quotient system. By Proposition~\ref{prop:cofinallinear}, it suffices to produce an coherent nontrivial sequence $\bar x \in (\mathbf Q_{\mathcal I}\! \restriction \! C)^{(n)}$. 
%Since $C$ is well-ordered in type $\omega_{n+1}$ by almost inclusion, we may think
%of $\mathbf Q_{\mathcal I}\restriction C$ as being indexed by $\omega_{n+1}$. 
To simplify notation let $Q_\xi = P_{a_\xi}/G_{a_\xi}$, for all $\xi < \omega_{n+1}$. Let us also write $p^\eta_\xi$ for the canonical projection from $Q_\eta$ to $Q_\xi$, for $\xi \leq \eta$. Therefore, we can identify the system $\mathbf Q_{\mathcal I} \! \restriction \! C$ with the system
\[
\mathbf Q= \{ Q_\xi; p^\eta_\xi \colon \xi \leq \eta < \omega_{n+1} \}. 
\]
Let $T$ be the tree $2^{< \omega_{n+1}}$ with the ordering being end-extension. 
% Let $\leq$ be the usual ordering on $T$, i.e., we let $s\leq t$ iff $s$ is an initial segment of $t$. 
For $s \in T$, we let $l(s)= \dom(s)$. We will define, for all $s\in T$, a sequence $\bar x^s$ such that: 
 
\begin{enumerate}
\item $\bar x^s \in \mathbf Q \restriction l(s)^{(n)}$ and is coherent
\item if $s\leq t$ then $\bar x^s$ is an initial segment of $\bar x^t$
\item\label{cond3phi} if ${\rm cof}(l (s))= \aleph_n$, and $\bar u\in \mathbf Q \restriction l(s)^{(n-1)}$ trivializes $\bar x^s$, then there is $\epsilon \in \{ 0, 1\}$ such that there is no extension of $\bar u$ in $\mathbf Q \restriction (l(s)+1)^{(n-1)}$ that trivializes $\bar x ^{s, \epsilon}$. 
\end{enumerate} 

We do this by induction on $l(s)$. If $l(s)$ is a limit ordinal and $\bar x ^t$ has been constructed, for all $t$ with $l(t) < l(s)$, we let $\bar x^s = \bigcup_{\alpha < l(s)} \bar x^{s \restriction \alpha}$. 

Consider now the successor case. Suppose $s\in T$ and let $\lambda = l(s)$. We have constructed the sequence $\bar x^s \in \mathbf Q \restriction \lambda^{(n)}$, and we need to extend it to two coherent sequences $\bar x^{s,\epsilon} \in \mathbf Q \restriction (\lambda +1)^{(n)}$, for $\epsilon \in \{ 0, 1\}$. For this we need to pick $x^{s,\epsilon}_{\bar \alpha, \lambda}\in Q_{\alpha_0}$, for all $\bar \alpha \in \lambda ^{(n-1)}$, and $\epsilon \in \{0,1\}$.
%This amounts to picking two sequence $\bar w^\epsilon\in \mathbf Q \restriction \lambda ^{(n-1)}$, for $\epsilon \in \{0,1\}$,
%and letting $x^{s,\epsilon}_{\bar \alpha, \lambda}= w^\epsilon_{\bar \alpha}$, for all $\bar \alpha \in \lambda^{(n-1)}$.
Given $\bar w \in \mathbf Q \restriction \lambda ^{(n-1)}$, we define the sequence $\bar x^s \ast_\lambda \bar w$ by letting, for $\bar \alpha = (\alpha_0, \ldots, \alpha_n) \in \lambda^{(n)}$,
\[
(\bar x^s \ast_\lambda \bar w)_{\bar \alpha}=
\begin{cases}
x^s_{\bar \alpha} & \text { if } \alpha_n\neq \lambda \\
w_{\alpha_0, \ldots, \alpha_{n-1}} & \text{ if } \alpha_n=\lambda.\end{cases}
\]
Observe that 
\[
\delta^{n+1}(\bar x^s \ast_\lambda \bar w) = \delta^{n+1}(\bar x^s)\ast_\lambda ( \delta^n(\bar w) + (-1)^n \bar x^s).
\]
%\medskip
\noindent
It follows that $\bar x^s \ast_\lambda \bar w$ is an coherent sequence in $\mathbf Q \restriction (\lambda +1)^{(n)}$ if and only if $(-1)^{n+1}\bar w$ trivializes $\bar x^s$. Since $\lambda$ is of cofinality $\leq \omega_n$, and the system $\mathbf Q$ is flasque, by Theorem~\ref{thm:flasquegoblot}, we can always find such a sequence $\bar w$. If $\cof(\lambda) < \aleph_n$ we let 
\[
\bar x^{s,0}= \bar x ^{s,1} = \bar x^s \ast_\lambda \bar w.
\]
 Let us now concentrate on the case $\cof(\lambda)= \aleph_n$. Observe that if two sequence $\bar w ^0$ and $\bar w ^1$ both trivialize $\bar x^s$ then $\bar w^0 - \bar w^1$ is a coherent sequence in $\mathbf Q \restriction \lambda ^{(n-1)}$. Conversely, if $\bar w$ trivializes $\bar x^s$ and $\bar z$ is a coherent sequence in $\mathbf Q \restriction \lambda ^{(n-1)}$ then $\bar w + \bar z$ also trivializes $\bar x^s$. By $\varphi_n (G)$ and the fact that $\cof(\lambda)=\aleph_n$, there is a nontrivial coherent sequence $\bar z \in \mathbf Q \restriction \lambda ^{(n-1)}$. We fix some $\bar w \in \mathbf Q \restriction \lambda ^{(n-1)}$ such that $(-1)^{n+1} \bar w$ trivializes $\bar x^s$, and let:
\[
\bar x^{s,0} = \bar x^s \ast_\lambda \bar w; \hspace{5mm} \bar x^{s,1} = \bar x^{s} \ast_\lambda (\bar w + \bar z). 
\] 
%\medskip

Let us check that the sequence $\bar x^{s,0}$ and $\bar x^{s,1}$ satisfy condition~\ref{cond3phi}. Consider first the case $n=1$. Suppose that $\bar u^\epsilon$ trivializes $x^{s,\epsilon}$, for $\epsilon \in \{ 0,1 \}$, and $\bar u^0 \restriction \lambda = \bar u^1 \restriction \lambda$. We would then have that $u^1_\lambda - u^0_\lambda$ trivializes $\bar z$, which would be a contradiction. 

Let us now suppose that $n>1$. Given $\bar u \in \mathbf Q \restriction (\lambda +1)^{(n-1)}$, we define a sequence $d_\lambda (\bar u) \in \mathbf Q \restriction \lambda^{(n-2)}$ by letting, for $\bar \alpha = (\alpha_0, \ldots, \alpha_{n-2}) \in \lambda^{(n-2)}$,
\[
d_\lambda (\bar u)_{\bar \alpha} = u_{\bar \alpha, \lambda}.
\]
Note that in any case we have $\bar u = \bar u \restriction \lambda^{(n-1)} \ast_{\lambda} d_\lambda (\bar u)$. Now observe also that:
\[
\delta^{n}(\bar u) = \delta^{n} (\bar u \restriction \lambda^{(n-1)})\ast_\lambda ( \delta^{n-1}(d_\lambda (\bar u)) + (-1)^{n-1}\bar u\restriction \lambda^{(n-1)}).
\]
Suppose now that $\bar u^\epsilon$ trivializes $\bar x^{s,\epsilon}$, for $\epsilon \in \{0,1\}$, and $\bar u^0 \restriction \lambda^{(n-1)} = \bar u^1 \restriction \lambda^{(n-1)}$. It follows that:
\[
 \delta^{n-1}(d_\lambda (\bar u^0)) + (-1)^{n-1}\bar u^0 \restriction \lambda^{(n-1)}= \bar w,
\]
and 
\[
 \delta^{n-1}(d_\lambda (\bar u^1)) + (-1)^{n-1}\bar u^1\restriction \lambda^{(n-1)} = \bar w + \bar z,
\]

%\medskip

\noindent 
Since $\bar u^0 \restriction \lambda^{(n-1)}= \bar u^1 \restriction \lambda^{(n-1)}$, we get that: 

\[
\delta^{n-1} (d_\lambda (\bar u^1- \bar u^0)) = \bar z,
\]
which contradicts our assumption that $\bar z$ is nontrivial. This completes the construction of the sequences $\bar x^s$, for $s \in T$. For $b \in 2^{\omega_{n+1}}$, we define:
\[
\bar x^b = \bigcup_{\lambda < \omega_{n+1}} \bar x ^{b \restriction \lambda }.
\]
Note that $\bar x^b \in \mathbf Q ^{(n)}$ and is coherent, for all $b$. We now use the principle ${\rm w}\diamondsuit(S^n_{n+1})$ to find $b \in 2^{\omega_{n+1}}$ such that $\bar x^b$ is nontrivial. Let
\[
D=\{\lambda < \omega_{n+1}\colon\omega_n\cdot\lambda=\lambda\}.
\]
Note that $D$ is a club in $\omega_{n+1}$. Suppose $s \in T$, let $\lambda = l(s)$, and suppose that $\lambda \in D$. We think of $s$ as coding a sequence $\bar u^s \in \mathbf Q \restriction \lambda^{(n-1)}$. The exact nature of the coding is not important, we only require that:
\begin{itemize}
\item if $\lambda \in D$ then every $\bar u \in \mathbf Q \restriction \lambda^{(n-1)}$ is coded by some $s\colon \lambda \to 2$,
\item if $s\leq t$ then $\bar u^s$ is an initial segment of $\bar u^t$. 
\end{itemize} 

\noindent If $b\colon \omega_{n+1} \to 2$, we let 
\[
\bar u^b = \bigcup_{\lambda \in D} \bar u^{b \restriction \lambda}.
\]
We now define a function $F\colon T \restriction D \to 2$ as follows. Given $s \in T \restriction D$, if $\bar u^s$ trivializes the sequence $\bar x^s$ we let $F(s) = \epsilon \in \{ 0,1\}$ be such that no extension of $\bar u^s$ in $\mathbf Q \restriction (l(s)+1)^{(n-1)}$ trivializes $\bar x^{s,\epsilon}$. Such an $\epsilon$ exists by \ref{cond3phi}. If $\bar u^s$ does not trivialize $\bar x^s$ we let $F(s)$ be an arbitrary element of $\{ 0, 1\}$. Now, by ${\rm w}\diamondsuit(S^n_{n+1})$ there is a function $g\colon \omega_{n+1} \to 2$ such that, for every $b\colon \omega_{n+1} \to 2$, the set 
\[
 \{ \lambda \in S^n_{n+1} \colon g(\lambda)= F( b \restriction \lambda) \} 
\]
is stationary. We claim that $\bar x^g$ is nontrivial. Indeed, suppose $\bar u$ trivializes $\bar x^g$ and find $b \in 2^{\omega_{n+1}}$ such that $\bar u ^b = \bar u$. Then by the guessing property of $g$ there is some $\lambda \in D$ such that $g(\lambda) = F(b\restriction \lambda)$. This means that $\bar u^{b \restriction \lambda}$ cannot be extended to a sequence in $\mathbf Q \restriction (\lambda +1)^{(n-1)}$ that trivializes $\bar x^g \restriction (\lambda +1)^{(n)}$, a contradiction. 
\end{proof} 

\begin{theorem}\label{A-nontrivial}
Suppose $\mathfrak b = \mathfrak d = \aleph_n$, and ${\rm w}\diamondsuit(S^k_{k+1})$ holds, for all $k < n$. 
Then $\lim^n \mathbf A\neq 0$. 
\end{theorem} 
\begin{proof}
Recall that the system $\mathbf A$ corresponds to $\mathbf G_{\mathcal I_0}$, where $\mathcal I_0$ is the ideal on $\omega \times \omega$ generated by sets of the form:
\[
a_f= \{ (i,j) \in \omega \times \omega \colon j \leq f(i) \},
\]
for all $f \in \omega^\omega$, and $G_{m,n}= \mathbb Z$, for all $(n,m)\in \omega^2$. The ideal $\mathcal I_0$ is generated by a tower of length $\omega_n$ precisely when $\mathfrak b = \mathfrak d = \aleph_n$. In this case, we have that $\phi_n(\mathbb Z)$ suffices to get $\lim^n \mathbf A\neq 0$. But by Proposition \ref{ideals-Z}, $\phi_1(\mathbb Z)$ holds, and by Lemma \ref{stepping-up} and our set theoretic hypotheses, so does $\phi_n(\mathbb Z)$.
\end{proof}

\section{The consistency result}\label{S.Cons}

We now turn to the consistency of the combination of combinatorial principles we need for our main result. 

\begin{definition} Given a regular cardinal $\kappa$ and a set $I$, we let ${\rm Fn}(I,2,\kappa)$ denote the poset of partial functions $p$ such that ${\rm dom}(p)\in [I]^{<\kappa}$ and $p\colon {\rm dom}(p)\to 2$. The order is reverse inclusion. 
\end{definition} 

 It is well known that ${\rm Fn}(I,2,\kappa)$ is $\kappa$-closed and has the $(2^{<\kappa})^+$-chain condition. 
 Thus, if $2^{<\kappa}= \kappa$ then forcing with ${\rm Fn}(I,2,\kappa)$ preserves cardinals. 
If $\lambda$ is a cardinal and $I = \lambda \times \kappa$, then ${\rm Fn}(I,2,\kappa)$ is the usual poset for adding
$\lambda$ many Cohen subsets of $\kappa$. If ${\rm GCH}$ holds in $V$ and ${\rm cof}(\lambda) >\kappa$,
then in the generic extension $V[G]$ by ${\rm Fn}(\lambda \times \kappa, 2, \kappa)$ we have $2^\kappa = \lambda$ (see \cite[IV.7]{Kunen.2011}).
The following lemma is fairly standard. We include a proof for completeness. 

\begin{lemma}\label{exhaustion} Let $\kappa$ be an uncountable regular cardinal. Suppose $\mathbb P$ is a $\kappa$-cc poset and $\mathbb Q$ is $\kappa$-closed. 
Let $(p,q)\in \mathbb P \times \mathbb Q$ and let $\tau$ be a $\mathbb P \times \mathbb Q$-name for an ordinal. 
Then there is a condition $r \leq q$ and a $\mathbb P$-name for an ordinal $\sigma$ such that $(p,r) \forces \tau = \sigma$.
\end{lemma} 

\begin{proof} 
We build inductively an antichain $(p_\xi \colon \xi < \delta)$ of conditions in $\mathbb P$ below $p$, a sequence of ordinals $(\alpha_\xi \colon \xi < \delta)$
and a decreasing sequence of conditions $(q_\xi \colon \xi < \delta)$ in $\mathbb Q$ below $q$. 
%The ordinal $\delta$ will be the stage at which the construction stops. We will show that $\delta < \kappa$. 
At stage $\gamma <\kappa$ we have built $(p_\eta \colon \xi <\gamma)$, $(\alpha_\xi \colon \xi < \gamma)$, and $(q_\xi \colon \xi <\gamma)$. 
%Since $\mathbb Q$ is $\kappa$-closed we can first find a lower bound $q'$ for the sequence $\{ q_\xi \colon \xi < \gamma\}$. 
We ask if there is a condition $p^* \leq p$ which is incompatible with all the $p_\xi$, for $\xi <\gamma$, an ordinal $\alpha^*$
and $q^*\leq q_\xi$, for all $\xi < \gamma$, such that $(p^*,q^*)\forces \tau = \check{\alpha}$. 
If so, we pick such witnesses $p^*$, $q^*$, and $\alpha^*$, and we let $p_\gamma = p^*$, $q_\gamma= q^*$, and $\alpha_\gamma = \alpha^*$. 
Since $\mathbb P$ is $\kappa$-cc the construction cannot continue for $\kappa$ many steps. 
Let $\delta$ be the stage at which it stops. 
Since $\mathbb Q$ is $\kappa$-closed we can find $r\in \mathbb Q$ such that $r\leq q_\xi$, for all $\xi <\delta$.
We let $\sigma = \{ (p_\xi, \check{\alpha}_\xi) \colon \xi < \delta\}$. 
Then $r$ and $\sigma$ are as required. 
\end{proof} 

\begin{lemma}\label{weak diamond}
Let $\kappa$ be an uncountable regular cardinal, and let $\lambda$ and $\mu$ be cardinals such that $\lambda^{<\kappa} < \mu$. Suppose $\mathbb P$ is a $\kappa$-cc poset of size $\lambda$. Let $\mathbb Q = {\rm Fn}(\mu \times \kappa, 2, \kappa)$. Then, for every stationary subset $S$ of $\kappa$ in $V$, ${\rm w}\diamondsuit(S)$ holds in the generic extension by $\mathbb P \times \mathbb Q$. 
\end{lemma} 

\begin{proof} 
We may assume that $\lambda\geq 2$, since otherwise forcing with $\mathbb P \times \mathbb Q$ is equivalent to forcing with $\mathbb Q$ and it is well-known that $\mathbb Q$ adds a regular $\diamondsuit(S)$-sequence. Let $\dot{F}$ be a $\mathbb P \times \mathbb Q$-name and suppose without loss of generality that the maximal condition in $\mathbb P \times \mathbb Q$ forces that $\dot{F}$ is a function from $2^{<\kappa}$ to $2$. Since $\mathbb P$ has the $\kappa$-cc and cardinality $\lambda$ there are at most $\lambda ^{<\kappa}$ canonical $\mathbb P$-names for elements of $2^{<\kappa}$. Since $\mathbb P \times \mathbb Q$ has the $2(^{<\kappa})^+$-cc we can find a subset $A$ of $\mu$ of size $\lambda^{<\kappa}$ such that, letting $\mathbb Q \restriction A$ denote the poset ${\rm Fn}(A \times \kappa, 2,\kappa)$, for every canonical $\mathbb P$-name $\tau$ for an element of $2^{<\kappa}$, $\dot{F}(\tau)$ is a $\mathbb P \times (\mathbb Q \restriction A)$-name. Since the cardinality of $A$ is less than $\mu$, we can find $\delta \in \mu \setminus A$. Let $\dot{g}_\delta$ be the canonical $\mathbb Q$-name for the $\alpha$-th Cohen subset of $\kappa$ added by $\mathbb Q$. We now claim that $\dot{g}_\delta$ is forced by the maximal condition in $\mathbb P \times \mathbb Q$ to be a guessing function for $\dot{F}$. 

So, fix a condition $(p,q) \in \mathbb P \times \mathbb Q$, a $\mathbb P\times \mathbb Q$ name $\dot{f}$ for a function from $\kappa$ to $2$, and a $\mathbb P\times \mathbb Q$ name $\dot{C}$ for a club in $\kappa$. We need to find some $(p^*,q^*) \leq (p,q)$ and $\alpha \in S$ such that 
\[
(p^*,q^*) \forces \check{\alpha}\in\dot{C} \text{ and }\dot{F}(\dot{f} \restriction \check{\alpha})=\dot{g}_\delta(\check{\alpha}).
\]
 We now proceed as in the standard proof that adding a Cohen subset of $\kappa$ adds a $\diamondsuit(S)$-sequence. To do this, we build an increasing continuous sequence $(\alpha_\xi \colon \xi<\kappa)$ of ordinals less than $\kappa$, a decreasing sequence of conditions $(q_\xi\colon \xi <\kappa)$ in $\mathbb Q$ below $q$, and two other sequences $(\sigma_\xi\colon \xi <\kappa)$ and $(\dot{f}_\xi \colon \xi <\kappa)$ of canonical $\mathbb P$-names such that: 

\begin{itemize} 
\item $\sigma_\xi$ is a $\mathbb P$-name for an ordinal $ < \kappa$ and is forced by $(p,q_{\xi+1})$ to be the next element of $\dot{C}$ above $\check{\alpha}_\xi$. Since $\mathbb P$ has the $\kappa$-cc we have $<\kappa$ many possible values for $\sigma_\xi$.
\item $\dot{f}_\xi$ is a $\mathbb P$-name for a a function from $\alpha_\xi$ to $2$ and $(p,q_{\xi+1}) \forces \dot{f} \restriction \check{\alpha}_\xi =\dot{f}_\xi$.
\end{itemize}

We can do this by the Lemma \ref{exhaustion}. Suppose we have done the construction up to $\xi$. We need to define the $\xi$-th elements of the above sequences. Consider first the successor case, i.e., $\xi= \eta +1$, for some $\eta$.
%We first ask the question what is the next point of $\dot{C}$ above \alpha_\eta$. 
By Lemma \ref{exhaustion} we can extend $q_\eta$ to $q'$ and find a canonical $\mathbb P$-name for an ordinal $\sigma_\xi$ such that $(p, q')$ forces that $\sigma_\xi$ is the next point of $\dot{C}$ above $\alpha_\eta$. Since $\mathbb P$ is $\kappa$-cc there are $<\kappa$ many possible values of $\sigma_\xi$. We then pick some $\alpha_{\xi}$ above all these values. Moreover, we ensure that $\alpha_\xi$ is at least $\sup \{ \alpha \colon (\delta,\alpha) \in \dom(q_\eta)\}$. Now, by Lemma \ref{exhaustion} again, we extend $q'$ to some $q_\xi$, and find a canonical $\mathbb P$-name $\dot{f}_\xi$ such that $(p,q_\xi)$ forces that $\dot{f}\restriction \check{\alpha}_\xi$ equals $\dot{f}_\xi$. At a limit stage $\xi$ we let 
\[
\alpha_\xi = \sup \{ \alpha_\eta \colon \eta < \xi \}, \,\, q_\xi = \bigcup \{ q_\eta \colon \eta < \xi \}\text{, and }\dot{f}_\xi = \bigcup \{ \dot{f}_\eta \colon \eta < \xi \}.
\]
 We also let $\sigma_\xi = \check{\alpha}_\xi$. Since $\dot{C}$ is forced to be a club in $\kappa$, we will have that $(p, q_\xi) \forces \check{\alpha}_\xi \in \dot{C}$. We continue this construction for $\kappa$ many steps. The set $\{ \alpha_\xi \colon \xi < \kappa\}$ will be a club in $\kappa$ and, since $S$ is stationary, we are going to reach a limit stage $\xi$ such that $\xi \in S$. Note that $(p, q_\xi)$ forces that $\dot{f} \! \restriction \! \check{\alpha}_\xi$ is equal to $\dot{f}_\xi$, which is a $\mathbb P$-name. Now, we can extend $(p,q_\xi)$ to some condition $(p^*,q')$ which decides the value of $\dot{F} (\dot{f}_\xi)$. Since $\dot{F}(\dot{f}_\xi)$ is a $\mathbb P \times (\mathbb Q \restriction A)$-name and $\delta \notin A$ we can do this without extending the $\delta$-th component of $q_\xi$. Note also that $\sup\{\alpha\colon(\delta,\alpha) \in \dom(q_\xi)\} \leq \alpha_\xi$. This means that $(p^*,q')$ decides the value of $\dot{F}(\dot{f}\! \restriction \! \check{\alpha}_\xi)$, but we have not yet made a commitment on the value of $\dot{g}_\delta$ at $\alpha_\xi$, i.e., $(\delta,\alpha_\xi)\notin \dom(q')$. So if this decision is equal to some $i \in \{ 0, 1\}$ we can now extend the $\delta$-th component of $q'$ by putting value putting this value $i$ at $\alpha_\xi$; in other words, we let $q^* = q' \cup \{ ((\delta, \alpha_\xi),i)\}$. Putting everything together, we have found a condition $(p^*,q^*)$ extending $(p,q)$ which forces $\check{\alpha}_\xi$ is in $\dot{C}$ and $\dot{g}_\delta(\alpha_\xi)$ equals to $\dot{F}(\dot{f} \! \restriction \! \check{\alpha}_\xi)$. This is what we needed to prove. 
\end{proof}

We recall the following well-known forcing notion $\mathbb H$ due to Hechler, \cite{Hechler}, but see \cite[IV.4.11]{Kunen.2011}.

\begin{definition}\label{Hechler} 
We say $p\in \mathbb H$ if $p =(s_p,f_p)$, where $s_p$ is a finite partial function from $\omega$ to $\omega$, and $f_p \in \omega^\omega$. We say that $q\leq p$ if $s_p \subseteq s_q$, and $s_q(n) \geq f_p(n)$, for all $n\in \dom(s_q) \setminus \dom(s_p)$. 
\end{definition} 

The forcing notion $\mathbb H$ is c.c.c., in fact, $\sigma$-centered, and adds a function $g\colon \omega \to \omega$ which dominates modulo finite sets all ground model functions in $\omega^\omega$. Given a cardinal $\kappa$ we let $\mathbb H_\kappa$ denote the finite support iteration of $\mathbb H$ of length $\kappa$. It is a standard fact that if $\kappa$ is uncountable then forcing with $\mathbb H_\kappa$ makes $\mathfrak b = \mathfrak d =\kappa$. 

Let us also recall that for an integer $k$, $S^k_{k+1}$ denotes the set $\{ \alpha < \omega_{k+1}\colon {\rm cof}(\alpha)= \omega_k\}$. 

\begin{theorem}\label{consistency} 
Let $n\geq 1$. Then it is relatively consistent with ${\rm ZFC}$ that $\mathfrak b= \mathfrak d = \aleph_n$, and ${\rm w}\diamondsuit(S^k_{k+1})$ holds, for all $k <n$. 
\end{theorem} 

\begin{proof} 
We start with a model $V$ satisfying ${\rm GCH}$. For $1\leq k\leq n$, let $\mathbb C_k = {\rm Fn}(\omega_{n+k}\times \omega_k, 2, \omega_k)$ and $\mathbb C = \mathbb H_{\omega_n} \times \mathbb C_1 \cdots \times \mathbb C_n$ be the product forcing notion. It is a standard fact (see \cite[V.2]{Kunen.2011}) that forcing with $\mathbb C$ preserves cofinalities. Let $G= G_0\times \cdots \times G_n$ be a generic over $\mathbb C$. Then in $V[G]$ we have $2^{\aleph_k}= \aleph_{n+k}$, for all $k\leq n$. By the product lemma and the fact that $\mathbb C_1 \times \ldots \times \mathbb C_n$ does not add reals, we can view the final model $V[G_0 \times G_1 \times \ldots \times G_n]$ as a generic extension over $V[G_1 \times \ldots \times G_n]$ by $\mathbb H_{\omega_n}$ as defined in that model. Therefore the final model satisfies $\mathfrak b = \mathfrak d = \aleph_n$. We need to check that ${\rm w}\diamondsuit(S^k_{k+1})$ holds in $V[G]$, for all $k <n$. Fix $k <n$. The forcing notion $\mathbb C_{k+2}\times \cdots \times \mathbb C_n$ is $\omega_{k +2}$-closed. Hence, if we let $W=V[G_{k+2}\times \cdots \times G_n]$, we have $2^{\omega_k}=\omega_{k+1}$ in $W$. Moreover, if we let $\mathbb P = \mathbb H_{\omega_n} \times \cdots \times \mathbb C_{k}$, $\mathbb Q = \mathbb C_{k+1}$, $\kappa = \omega_{k+1}$, $\lambda = \omega_{n+k}$, $\mu = \omega_{n+k+1}$, and $S= S^k_{k+1}$, then the assumptions of Lemma \ref{weak diamond} are satisfied in $W$. By the product lemma of forcing, $V[G]$ can be viewed as a generic extension of $W$ by $\mathbb P \times \mathbb Q$. It follows that ${\rm w}\diamondsuit(S^k_{k+1})$ holds in $V[G]$, as required. 
\end{proof} 

Finally, from Theorem \ref{consistency} and Theorem \ref{A-nontrivial} we conclude the main result of this paper. 

\begin{theorem}\label{main-theorem} 
Let $n \geq 1$. Then it is relatively consistent with ${\rm ZFC}$ that $\lim^n \mathbf A\neq 0$.\qed
\end{theorem}

\bibliographystyle{plain} %{amsalpha}
\bibliography{library}

\begin{thebibliography}{10}

\bibitem{BanBergMoore}
N.~Bannister, J.~Bergfalk, and J.T. Moore.
\newblock On the additivity of strong homology for locally compact separable
  metric spaces.
\newblock arXiv:2008.13089, 2020.

\bibitem{Bekkali1991}
M.~Bekkali.
\newblock {\em {Topics in set theory. Lebesgue measurability, large cardinals,
  forcing axioms, rho-functions}}, volume 1476.
\newblock Berlin etc.: Springer-Verlag, 1991.

\bibitem{Bergfalk2017}
J.~Bergfalk.
\newblock Strong homology, derived limits, and set theory.
\newblock {\em Fund. Math.}, 236(1):71--82, 2017.

\bibitem{Bergfalk.Alephs}
J.~{Bergfalk}.
\newblock The first omega alephs: from simplices to trees of trees to higher
  walks.
\newblock arXiv:2008.03386, 2020.

\bibitem{Bergfalk2021simutaneously}
J.~{Bergfalk}, M.~{Hru\v{s}\'ak}, and C.~{Lambie-Hanson}.
\newblock Simutaneously vanishing higher derived limits without large
  cardinals.
\newblock arXiv:2102.06699, 2021.

\bibitem{BergfalkLH}
J.~{Bergfalk} and C.~{Lambie-Hanson}.
\newblock Simultaneously vanishing higher derived limits.
\newblock arXiv:1907.11744, 2019.

\bibitem{Devlin-Shelah}
K.~J. Devlin and S.~Shelah.
\newblock A weak version of {$\Diamond$} which follows from {$2^{\aleph
  _{0}}<2^{\aleph _{1}}$}.
\newblock {\em Israel J. Math.}, 29(2-3):239--247, 1978.

\bibitem{Dow-Simon-Vaughan}
A.~Dow, P.~Simon, and J.~E. Vaughan.
\newblock Strong homology and the proper forcing axiom.
\newblock {\em Proc. Amer. Math. Soc.}, 106(3):821--828, 1989.

\bibitem{Goblot}
R.~Goblot.
\newblock Sur les d\'{e}riv\'{e}s de certaines limites projectives.
  {A}pplications aux modules.
\newblock {\em Bull. Sci. Math. (2)}, 94:251--255, 1970.

\bibitem{Hechler}
S.~H. Hechler.
\newblock On the existence of certain cofinal subsets of {$^{\omega }\omega $}.
\newblock In {\em Axiomatic set theory ({P}roc. {S}ympos. {P}ure {M}ath.,
  {V}ol. {XIII}, {P}art {II}, {U}niv. {C}alifornia, {L}os {A}ngeles, {C}alif.,
  1967)}, pages 155--173, 1974.

\bibitem{Hilton-Stammbach}
P.~J. {Hilton} and U.~{Stammbach}.
\newblock {\em {A course in homological algebra. 2nd ed}}, volume~4.
\newblock New York, NY: Springer, 2nd ed. edition, 1997.

\bibitem{Kunen.2011}
K.~Kunen.
\newblock {\em Set theory}, volume~34 of {\em Studies in Logic (London)}.
\newblock College Publications, London, 2011 (revised edition 2013).

\bibitem{Mardesic.Book}
S.~Marde\v{s}i\'{c}.
\newblock {\em Strong shape and homology}.
\newblock Springer Monographs in Mathematics. Springer-Verlag, Berlin, 2000.

\bibitem{MardesicPrasolov}
S.~Marde\v{s}i\'{c} and A.V. Prasolov.
\newblock Strong homology is not additive.
\newblock {\em Trans. Amer. Math. Soc.}, 307(2):725--744, 1988.

\bibitem{Mitchell72}
B.~Mitchell.
\newblock Rings with several objects.
\newblock {\em Advances in Math.}, 8:1--161, 1972.

\bibitem{Mitchell73}
B.~Mitchell.
\newblock {The cohomological dimension of a directed set}.
\newblock {\em {Can. J. Math.}}, 25:233--238, 1973.

\bibitem{shelah_2017}
S.~Shelah.
\newblock {\em Proper and Improper Forcing}.
\newblock Perspectives in Logic. Cambridge University Press, 2 edition, 2017.

\bibitem{Todorcevic2007}
S.~{Todor\v{c}evi\'c}.
\newblock {\em {Walks on ordinals and their characteristics}}, volume 263.
\newblock Basel: Birkh\"auser, 2007.

\end{thebibliography}
\end{document}